\documentclass{article}
\usepackage[utf8]{inputenc}

\usepackage{amsfonts}
\usepackage{amssymb}
\usepackage{amsmath}
\usepackage{amsthm}
\usepackage{mathtools}
\usepackage{cleveref}
\usepackage{url}
\usepackage{wasysym}
\usepackage{tikz}
\usepackage{tikz-cd}
\usepackage{adjustbox}
\usepackage{stmaryrd}
\usepackage{enumitem}
\usepackage{subcaption}
\usepackage{graphicx}
\usepackage{afterpage}
\usepackage{changepage}
\usepackage{authblk}

\usepackage[
    backend=biber,
    style=alphabetic,
    sorting=nyt,
    doi=false,
    url=false,
    isbn=false,
    maxbibnames=10,
]{biblatex}
\DeclareFieldFormat[misc]{title}{``#1"}

\theoremstyle{plain}
    \newtheorem{theorem}{Theorem}
    \newtheorem{proposition}{Proposition}[section]
    \newtheorem{corollary}[proposition]{Corollary}
    \newtheorem{lemma}[proposition]{Lemma}

\theoremstyle{definition}
    \newtheorem{definition}[proposition]{Definition}
    \newtheorem{algorithm}[proposition]{Algorithm}
    
    \newtheorem{example}[proposition]{Example}
    \newtheorem{question}[proposition]{Question}

\theoremstyle{remark}
	\newtheorem{remark}[proposition]{Remark}%

\newtheoremstyle{repeat-theorem}
    {\topsep}{\topsep}              %%% space between body and thm
    {\itshape}                      %%% Thm body font
    {}                              %%% Indent amount (empty = no indent)
    {\bfseries}                     %%% Thm head font
    {.}                             %%% Punctuation after thm head
    { }                             %%% Space after thm head
    {\thmname{#1}\thmnote{ \bfseries #3}}%%% Thm head spec
    
\theoremstyle{repeat-theorem}
    \newtheorem{repeat-theorem}{Theorem}
    \newtheorem{repeat-proposition}{Proposition}
    \newtheorem{repeat-corollary}{Corollary}
    
\numberwithin{equation}{section}

\newcommand{\ZZ}{\mathbb{Z}}
\newcommand{\QQ}{\mathbb{Q}}
\newcommand{\RR}{\mathbb{R}}

\newcommand{\FF}{\mathbb{F}}

\newcommand{\id}{\mathit{id}}
\newcommand{\isom}{\cong}

\renewcommand{\epsilon}{\varepsilon}

\DeclareMathOperator{\Hom}{Hom}

\DeclareMathOperator{\Tor}{Tor}
\DeclareMathOperator{\Ima}{Im}
\DeclareMathOperator{\Ker}{Ker}

\DeclareMathOperator{\fchar}{char}

\DeclareMathOperator{\gr}{gr}

\makeatletter
\newcommand*\fatcdot{\mathpalette\fatcdot@{.5}}
\newcommand*\fatcdot@[2]{\mathbin{\vcenter{\hbox{\scalebox{#2}{$\m@th#1\bullet$}}}}}
\makeatother

\DeclarePairedDelimiterX\set[1]\lbrace\rbrace{\,\setaux#1\,}
 \def\setaux#1|{#1\;\delimsize\vert\;}

\newcommand{\Kh}{\mathit{Kh}}

\newcommand{\Lee}{\mathit{Lee}}
\newcommand{\BN}{\mathit{BN}}

\newcommand{\ca}{\alpha}
\newcommand{\cb}{\beta}

\newcommand{\cz}{\zeta}

\renewcommand{\tilde}{\widetilde}

\renewcommand{\ss}{\mathit{ss}}
\newcommand{\ssr}{\tilde{\ss}}

\title{
    A family of slice-torus invariants from the divisibility of Lee classes
}

\author{Taketo Sano and Kouki Sato}

\newcommand{\addresses}{{
  \bigskip
  Taketo Sano, \textsc{RIKEN iTHEMS, Wako, Saitama 351-0198, Japan }\par\nopagebreak
  \textit{E-mail address}: \url{taketo.sano@riken.jp}

  \medskip

  Kouki Sato, \textsc{Meijo University,Tempaku, Nagoya 468-8502, Japan}\par\nopagebreak
  \textit{E-mail address}: \url{satokou@meijo-u.ac.jp}
}}

\begin{document}
    \maketitle
    \begin{abstract}
    We give a family of slice-torus invariants $\ssr_c$, each defined from the $c$-divisibility of the reduced Lee class in a variant of reduced Khovanov homology, parameterized by prime elements $c$ in any principal ideal domain $R$. For the special case $(R, c) = (F[H], H)$ where $F$ is any field, we prove that $\ssr_c$ coincides with the Rasmussen invariant $s^F$ over $F$. 
    Compared with the unreduced invariants $\ss_c$ defined by the first author in a previous paper, we prove that $\ss_c = \ssr_c$ for $(R, c) = (F[H], H)$ and $(\ZZ, 2)$. However for $(R, c) = (\ZZ, 3)$, computational results show that $\ss_3$ is not slice-torus, which implies that it is linearly independent from the reduced invariants, and particularly from the Rasmussen invariants. 
\end{abstract}
    
    % \setcounter{tocdepth}{2}
    % \tableofcontents
    
    \section{Introduction}
\label{sec:intro}

Rasmussen's \textit{$s$-invariant} is an integer valued knot invariant obtained from a variant of Khovanov homology \cite{Khovanov:2000,Lee:2005,Rasmussen:2010}. Its major applications are Rasmussen's alternative proof of the Milnor conjecture \cite{Milnor:1968} and Piccirillo's proof of the non-sliceness of the Conway knot \cite{Piccirillo:2020}. Both problems arose in the intersection of knot theory and 4-dimensional topology and remained unsolved for decades. The $s$-invariant belongs to a class called the \textit{slice-torus invariants} \cite{Livingston:2004,Lewark:2014}. The existence of a slice-torus invariant is itself nontrivial, for it immediately implies the Milnor conjecture. 

\begin{definition}
\label{def:slice-torus-inv}
    A \textit{slice-torus invariant} $\nu$ is an abelian group homomorphism
    \[
        \nu: \mathit{Conc}(S^3) \rightarrow \RR,
    \]
    satisfying the following two conditions: 
    \begin{enumerate}[leftmargin=3.5\parindent]
        \item[(Slice)] \ $|\nu(K)| \leq 2g_4(K)$ for any knot $K$,
        \item[(Torus)] \ $\nu(T_{p, q}) = (p - 1)(q - 1)$ for the positive $(p, q)$-torus knot $T_{p, q}$.
    \end{enumerate}
    Here $\mathit{Conc}(S^3)$ denotes the smooth concordance group of knots in $S^3$, and $g_4(K)$ the slice genus of a knot $K$. 
\end{definition}

The $s$-invariant is generalized over any field $F$, and is known that $s^\QQ,\ s^{\FF_2},\ s^{\FF_3}$ are distinct \cite{MTV:2007,LS:2014_rasmussen,Lewark-Zibrowius:2022,Schuetz:2022}. Other examples of slice-torus invariants are: (i) the $\tau$-invariant obtained from knot Floer homology \cite{OS:2003}, (ii) the $s_n$-invariants $(n \geq 2)$ from $\mathfrak{sl}_n$ Khovanov--Rozansky homologies \cite{Wu:2009,Lobb:2009,Lobb:2012}, (iii) the $\tau^\#$-invariant from framed instanton Floer homology \cite{Baldwin-Sivek:2021}, and (iv) the $\tilde{s}$-invariant from equivariant singular instanton Floer homology \cite{DISST:2022}. Studies on general slice-torus invariants are given in \cite{Cavallo:2020,Feller:2022}.

In this paper we introduce a family of slice-torus invariants, each defined from the divisibility of the \textit{reduced Lee class} $\tilde{\ca}_c(D)$.

\begin{theorem}
\label{mainthm:1}
    For each prime element $c$ in a principal integral domain $R$, the value defined by 
    \[
        \ssr_c(K) = 2d_c(D) + w(D) - r(D) + 1
    \]
    is a slice-torus invariant. Here $K$ is a knot with diagram $D$, $d_c(D)$ the $c$-divisibility of the reduced Lee class $\tilde{\ca}_c(D)$, $w(D)$ the writhe and $r(D)$ the number of Seifert circles of $D$. 
\end{theorem}

The definition of $\ssr_c$ is formally identical to that of the invariant $\ss_c$ given in \cite{Sano:2020}, which is \textit{not} proved to be slice-torus in general. Arguments of \cite{Sano:2020} run in parallel, and from the simplicity that the reduced homology $\tilde{H}_c(D)$ has rank $1$, the slice-torus properties for $\ssr_c$ follow straightforwardly. The following theorem states that our family contains the Rasmussen invariants $s^F$.

\begin{theorem}
\label{mainthm:2}
    For $(R, c) = (F[H], H)$ where $F$ a field of any characteristic, we have $\ssr_c = \ss_c = s^F$.
\end{theorem}

Now the question becomes whether $\ssr_c$ and $\ss_c$ are distinct, and whether they contain an invariant that is independent from $s^F$. For $(R, c) = (\ZZ, 2)$, we can similarly prove that $\ssr_2 = \ss_2$. However for $(R, c) = (\ZZ, 3)$, computational results show that there are knots $K$ such that $\ss_3(K) \neq -\ss_3(m(K))$, where $m$ denotes the mirror. These are knots that satisfy $s^\QQ(K) \neq s^{\FF_3}(K)$, found by Schuetz \cite{Schuetz:2022}. Thus we have

\begin{theorem}
\label{mainthm:3}
    The unreduced invariant $\ss_3$ is not a slice-torus invariant. Thus $\ss_3$ is linearly independent from $\ssr_c$, and particularly from the Rasmussen invariants.
\end{theorem}

It remains open whether the reduced invariants $\ssr_c$ contain one that is independent from $s^F$. 

\medskip

The advantage of reformulating $s^F$ as in \Cref{mainthm:2} is that, compared to the original definition by the quantum filtration, it enables us to treat the invariant more directly via the Lee class.
For example, the mysterious phenomena that $s^\QQ \neq s^{\FF_2}$ for some knots can be observed directly by considering the Lee class for $(R, c) = (\ZZ[H], H)$, and then relating it to the cases $(R, c) = (\QQ[H], H)$ and $(\FF_2[H], H)$. This is discussed in \Cref{subsec:ZH-theory}.

For another example, we obtain another reformulation of $s^F$ which is analogous to the one for $s^\QQ$ given by Kronheimer--Mrowka in \cite{KM:2011}. Note that the formula is more simple, since we are considering the reduced theory.

\begin{proposition}
    Let $K$ be a knot, and $S$ a connected cobordism from the unknot $U$ to $K$. Consider the induced map between the reduced homology groups
    \[
        \phi_S: \tilde{H}_H(U; F[H]) = F[H] \rightarrow \tilde{H}_H(K; F[H]).
    \]
    Then
    \[
        s^F(K) = 2d_c(z) + \chi(S),
    \]
    where $z = \phi_S(1)$, $d_c(z)$ the $c$-divisibility of $z$ (modulo torsion), and $\chi(S)$ the Euler characteristic of $S$. (Stated more precisely in \Cref{prop:s_c-by-cobordism}.)
\end{proposition}

As a future prospect, we expect that our reformulation of $s$ reveals its connection with the invariant $s^\#$, a gauge theoretic analogue of $s$ introduced in \cite{KM:2011}. It was originally thought that $s$ and $s^\#$ are equal, but turned out to be distinct; in fact $s^\#$ is \textit{not} slice-torus \cite{Gong:2021}. A reduced version invariant $\tilde{s}^\#$ is introduced in \cite{DISST:2022}%
\footnote{
    The invariant is denoted $\tilde{s}$ in the paper, but here we put $\#$ for distinguishability.
}%
, which is defined by the divisibility of the \textit{special cycles} in the reduced equivariant singular instanton Floer complex $\tilde{C}^\#$. The invariant $\tilde{s}^\#$ is proved to be slice-torus, and approximates $s^\#$ as 
\[
    s^\#(K) = 2\tilde{s}^\#(K) - \tilde{\epsilon}(K)
\]
with $\tilde{\epsilon}(K) \in \{0, \pm 1\}$. As instanton gauge theory is deeply connected to Khovanov theory \cite{KM:unknot,KM:2014,KM:2021}, we believe that there is also a reason for the similarity between the two unreduced-reduced pairs of invariants.

This paper is organized as follows. In \Cref{sec:khovanov} we briefly review the basics of Khovanov homology theory, including the construction of the Lee classes and results obtained in \cite{Sano:2020}. In \Cref{sec:module-str} we setup the foundation for the reduced theory in general. The reduced Lee classes are defined therein. In \Cref{sec:divisibility} the invariant $\ssr_c$ is defined as the $c$-divisibility of the reduced Lee class. Its coincidence with $s^F$ is proved and a classification result is given. In \Cref{sec:unreduced} we compare $\ssr_c$ with the unreduced counterpart $\ss_c$. Finally in \Cref{sec:computations} we briefly explain the algorithm of our computer program and list several computational results. Considerations for the case $(R, c) = (\ZZ[H], H)$ is also given therein.

\setcounter{theorem}{0}
\setcounter{conjecture}{0}

    \subsection*{Acknowledgements}

We thank Dirk Sch{\"u}tz for kindly providing us the input knot data and giving us many useful insights. We thank Naoya Umezaki, Melissa Zhang and Tomohiro Asano for useful discussions. We thank Yuto Horikawa and Toru3 for helping us developing the computer program. 
The work of TS was supported by JSPS KAKENHI Grant Numbers 23K12982, RIKEN iTHEMS Program and academist crowdfunding.

    \section{Khovanov homology}
\label{sec:khovanov}

% The contents of \Cref{subsec:lee-classes,subsec:reduction-of-params} are mostly given in \cite{Sano:2020}. 

Throughout this paper we work in the smooth category, and assume that all links and link diagrams are oriented. 

\subsection{Khovanov homology} \label{subsec:khovanov}

\begin{definition}
    Let $R$ be a commutative ring with unity. A \textit{Frobenius algebra} over $R$ is a quintuple $(A, m, \iota, \Delta, \epsilon)$ such that: 
    \begin{enumerate}
        \item $(A, m, \iota)$ is an associative $R$-algebra with multiplication $m$ and unit $\iota$,
        \item $(A, \Delta, \epsilon)$ is a coassociative $R$-coalgebra with comultiplication $\Delta$ and counit $\epsilon$,
        \item the Frobenius relation holds: 
        \[
            \Delta \circ m = (\id \otimes m) \circ (\Delta \otimes \id) = (m \otimes \id) \circ (\id \otimes \Delta).
        \]
    \end{enumerate}
\end{definition}

% The \textit{twisting} of $A$ by an invertible element $\theta \in R$ is another Frobenius algebra $A_\theta = (A, m, \iota, \Delta_\theta, \epsilon_\theta)$ whose algebra structure is the same as $A$ but the coalgebra structure is twisted as
% \[
%     \Delta_\theta(x) = \Delta(\theta^{-1}x),
%     \quad 
%     \epsilon_\theta(x) = \epsilon(\theta x).
% \]

Let $R$ be a commutative ring with unity, and $h, t$ be elements in $R$. Define $A_{h, t} = R[X]/(X^2 - hX - t)$. $A_{h, t}$ is endowed a Frobenius algebra structure as follows: the $R$-algebra structure is inherited from $R[X]$. Regarding $A_{h, t}$ as a free $R$-module with basis $\{1, X\}$, the counit $\epsilon: A_{h, t} \rightarrow R$ is defined by
\[
    \epsilon(1) = 0,\quad
    \epsilon(X) = 1.
\]
Then the comultiplication $\Delta$ is uniquely determined so that $(A_{h, t}, m, \iota, \Delta, \epsilon)$ becomes a Frobenius algebra. Explicitly, the operations $m$ and $\Delta$ are given by 
\begin{equation}
\label{eq:1X-operations}
    \begin{gathered}
        m(1 \otimes 1) = 1, \quad 
        m(X \otimes 1) = m(1 \otimes X) = X, \quad
    	m(X \otimes X) = hX + t, \\
        \Delta(1) = X \otimes 1 + 1 \otimes X - h (1 \otimes 1), \quad 
    	\Delta(X) = X \otimes X + t (1 \otimes 1).
    \end{gathered}
\end{equation}

Given a link diagram $D$, a complex $C_{h, t}(D; R)$ over $R$ is defined by the construction given in \cite{Khovanov:2000}, except that the defining Frobenius algebra $A = R[X]/(X^2)$ is replaced by $A_{h, t} = R[X]/(X^2 - hX - t)$. 

\begin{definition}
    The complex $C_{h, t}(D; R)$ is called the \textit{Khovanov complex}, and its homology denoted $H_{h, t}(D; R)$ is called the \textit{Khovanov homology} of $D$ (with respect to the triple $(R, h, t)$).
\end{definition}

% In addition to the homological grading, the complex $C_{h, t}(D)$ is given a secondary grading called the \textit{quantum grading}. Depending on the degrees of $h$ and $t$, the complex $C_{h, t}(D)$ and the homology $H_{h, t}(D)$ admits a bigrading or a filtration. 
Recall that Khovanov's original theory \cite{Khovanov:2000} is given by $(h, t) = (0, 0)$, Lee's theory \cite{Lee:2005} by $(h, t) = (0, 1)$, and (the filtered version of) Bar-Natan's theory \cite{BarNatan:2004} by $(h, t) = (1, 0)$. The universal one among all triples $(R, h, t)$ is given by $X^2 - hX - t$ over $R = \ZZ[h, t]$, which is called the \textit{$U(2)$-equivariant theory} (see \cite{Khovanov:2004,Khovanov:2022}).
The following proposition justifies referring to the isomorphism class of $H_{h, t}(D)$ as the \textit{Khovanov homology} of the corresponding link $L$.

\begin{proposition}[{\cite[Theorem 1]{BarNatan:2004}, \cite[Proposition 6]{Khovanov:2004}}]
\label{prop:Kh-invariance}
    Suppose $D, D'$ are link diagrams related by a Reidemeister move. Then there is a chain homotopy equivalence
    \[
        \rho: C_{h, t}(D) \rightarrow C_{h, t}(D').
    \]
    % Moreover this map respects the bigrading or the filtration when $C_{h, t}$ admits such structures.
\end{proposition}

The explicit descriptions of $\rho$ for the three Reidemeister moves are given in \cite[Section 4.3]{BarNatan:2004}. We call each of these maps an \textit{R-move map}. We introduce a few more terms and notations that will be used throughout this paper. 

\begin{definition}
    A \textit{state} $u$ of a diagram $D$ is an assignment of $0$ or $1$ to each crossing of $D$. 
\end{definition}

When a total ordering of the crossings of $D$ is given, a state $u$ is identified with an element $u \in \{0, 1\}^n$. We denote by $D(u)$ the crossingless diagram obtained from $D$ by resolving all crossings accordingly. 

\begin{definition}
    For an arbitrary set $S$, an \textit{$S$-enhanced state} of a diagram $D$ is a pair $x = (u, a)$ such that $u$ is a state of $D$ and $a$ is an assignment of an element in $S$ to each circle of $D(u)$. 
\end{definition}

When $S$ is a subset of $A_{h, t}$, an $S$-enhanced state is identified with an element of $C_{h, t}(D)$ by the corresponding tensor product of the elements of $S$. In particular for $S = \{1, X\} \subset A_{h, t}$, the set of all $1X$-enhanced states of $D$ forms a basis of $C_{h, t}(D)$, and is called the \textit{standard generators} of $C_{h, t}(D)$.
    \subsection{Lee classes}
\label{subsec:lee-classes}

Lee \cite{Lee:2005} proved that for any link diagram $D$, its $\QQ$-Lee homology has dimension $2^{|D|}$ and is generated by the explicitly constructed \textit{Lee classes} (or the \textit{canonical classes}, as called in \cite{Rasmussen:2010}). 
% Here $|D|$ denotes the number of components of $D$. 
Their argument is generalized in \cite{MTV:2007} and in \cite{Turner:2020}, provided that the Frobenius algebra $A_{h, t}$ is diagonalizable. Here we give an equivalent condition in terms of the defining quadratic polynomial of $A_{h, t}$.

\begin{definition} \label{def:factorable}
    We say $(R, h, t)$ is a \textit{factorable triple} (or simply \textit{is factorable}) if the quadratic polynomial $X^2 - hX - t \in R[X]$ factors as linear polynomials. 
\end{definition}

Khovanov's original theory $(h, t) = (0, 0)$, Lee's theory $(h, t) = (0, 1)$ and Bar-Natan's theory $(h, t) = (1, 0)$ are examples of such triples. The universal theory among all factorable triples is given by $X^2 - hX - t = (X - a)(X - b)$ over $R = \ZZ[a, b]$, which is called the \textit{$U(1) \times U(1)$-equivariant theory} (see \cite{Khovanov:2022}). 

For the remainder of this section we assume $(R, h, t)$ is factorable. Fix two roots $a, b \in R$ of $X^2 - hX - t$ and put $c = b - a$. Here $(h, t)$ and $(a, b, c)$ are related as
\[
    h = a + b, \quad t = -ab,\quad c^2 = h^2 + 4t.
\]
Define two elements
% \footnote{The pair $(X_a, X_b)$ corresponds to $(\mathbf{a}, \mathbf{b})$ in \cite{Lee:2005} and to $(X_{\text{\textcircled{\tiny{1}}}}, X_{\text{\textcircled{\tiny{2}}}})$ in \cite{Khovanov:2022}.} 
in $A_{h, t}$ by
\[
    X_a = X - a,\quad 
    X_b = X - b.
\]
Then we have 
\begin{equation}
\label{eq:ab-operations}
    \begin{gathered}
    	m(X_a \otimes X_a) = cX_a,\ m(X_b \otimes X_b) = -cX_b, \\
    	m(X_a \otimes X_b) = m(X_b \otimes X_a) = 0,\\
    	\Delta(X_a) = X_a \otimes X_a,\ \Delta(X_b) = X_b \otimes X_b.
    \end{gathered}
\end{equation}
Thus the operations $m$ and $\Delta$ are diagonalized in the submodule spanned by $\{X_a, X_b\}$. Moreover, from
\[
    \begin{pmatrix}
        X_a & X_b
    \end{pmatrix}
    =
    \begin{pmatrix}
        1 & X
    \end{pmatrix}
    \begin{pmatrix}
        -a & -b \\
        1 & 1
    \end{pmatrix},
\]
we see that $\{X_a, X_b\}$ form a basis of $A_{h, t}$ if and only if $c = b - a$ is invertible in $R$. Thus the Frobenius algebra $A_{h, t}$ is diagonalizable\footnote{As defined in \cite{Turner:2020}, a rank $2$ Frobenius algebra is $A$ over $R$ is \textit{diagonalizable} if there exists a basis $\{e_1, e_2\}$ of $A$ such that $e_1^2 = e_1,\ e_2^2 = e_2$ and $e_1 e_2 = 0$.  The equivalence (for one direction) can be seen by putting $e_1 = c^{-1}X_a,\ e_2 = -c^{-1}X_b$.} if and only if $(R, h, t)$ is factorable and $c$ is invertible in $R$. 
% Note that $c = 0$ for Khovanov's theory, $c = 2$ for Lee's theory, and $c = 1$ for Bar-Natan's theory. 

\begin{definition}
    For a link $L$, let $O(L)$ be the set of all orientations on the underlying unoriented link of $L$. The set $O(D)$ for a link diagram $D$ is defined likewise.
\end{definition}

For each $o \in O(L)$, let $L_o$ denote the link $L$ with its orientation replaced by $o$. A diagram $D_o$ is defined likewise for each $o \in O(D)$.

\begin{figure}[t]
    \centering
    \includegraphics[scale=0.35]{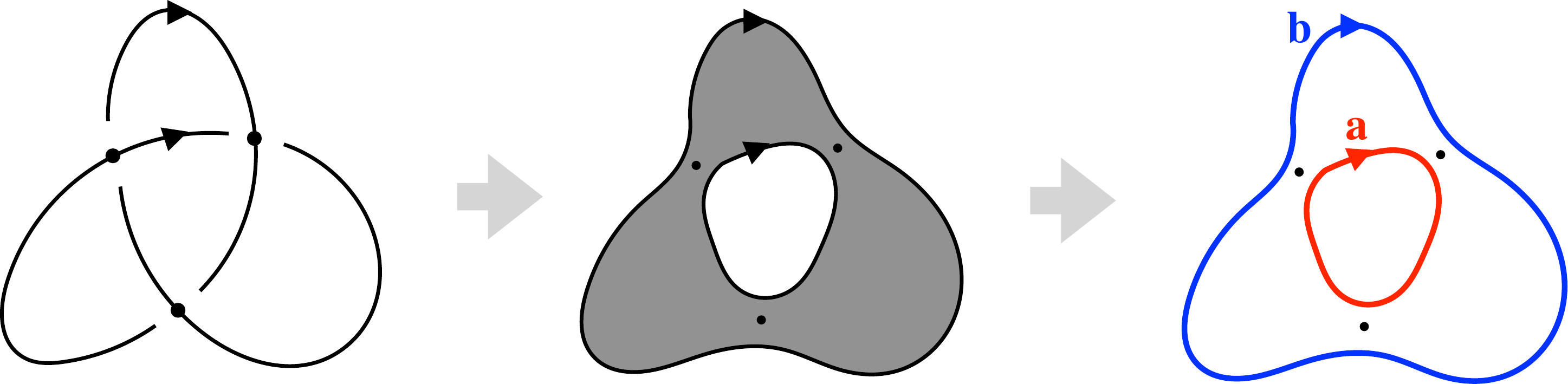}
    \caption{Construction of the Lee cycle.}
    \label{fig:ab}
\end{figure}

\begin{algorithm} \label{algo:ab-coloring}
    Given a link diagram $D$, the \textit{$ab$-coloring} on its Seifert circles are obtained as follows: separate $\mathbb{R}^2$ into regions by the Seifert circles of $D$, and color the regions in the checkerboard fashion, with the unbounded region colored white. For each Seifert circle, let it inherit the orientation from $D$. Assign to it $a$ if it sees a black region to the left, otherwise $b$.
\end{algorithm}

For any oriented link diagram $D$, there is a unique state $u$ that gives the orientation preserving resolution of $D$, so that $D(u)$ consists of the Seifert circles of $D$. The $ab$-coloring on $D$ determines a unique $X_aX_b$-enhanced state $\ca(D) \in C_{h, t}(D)$ given by the corresponding tensor products of $X_a$ and $X_b$. Similarly for each $o \in O(D)$, we obtain an $X_aX_b$-enhanced state $\ca(D, o) \in C_{h, t}(D)$ from the $ab$-coloring on $D_o$\footnote{In \cite{Rasmussen:2010} the element $\ca(D, o)$ is denoted $\mathfrak{s}_o$.}. 

\begin{definition}
    The above constructed elements $\ca(D, o)$ are called the \textit{Lee cycles} of $D$. 
\end{definition}

For any diagram $D$ there are $2^{|D|}$ distinct Lee cycles. For convenience we also write $\cb(D, o)$ for $\ca(D, -o)$. Then $\cb(D, o)$ is obtained by flipping all $X_a$'s and $X_b$'s for the tensor factors of $\ca(D, o)$.

\begin{example}
    For the diagram $D$ given in \Cref{fig:ab}, the Lee cycles are given by $\ca(D) = X_b \otimes X_a$ and $\cb(D) = X_a \otimes X_b$, where the first factor corresponds to the outer circle. 
\end{example}

\begin{proposition}
    $\ca(D, o)$ is indeed a cycle, i.e.\ $d\ca(D, o) = 0$.
\end{proposition}

\begin{proof}
    From the procedure of the $ab$-coloring, we see that each crossing connects differently colored Seifert circles of $D_o$, so all outgoing edge maps annihilates $\ca(D, o)$. 
\end{proof}

By abuse of notation, we write $\ca(D, o)$ and $\cb(D, o)$ for the corresponding homology classes. The following \Cref{prop:ab-gen,prop:alpha-homol-gr} generalize \cite[Theorem 4.2]{Lee:2005}, which are proved in \cite[Theorem 4.2]{Turner:2020} and in \cite[Proposition 2.9]{Sano:2020}. 

\begin{proposition} \label{prop:ab-gen}
	If $c$ is invertible, then $H_{h, t}(D; R)$ is freely generated over $R$ by the Lee classes. In particular $H_{h, t}(D; R)$ has rank $2^{|D|}$.
\end{proposition}

\begin{proposition} \label{prop:alpha-homol-gr}
    Let $D$ be an $l$-component link diagram and $D_1, \cdots, D_l$ be the component diagrams. For any orientation $o$ on the underlying unoriented diagram of $D$, let $I \subset \{1, \ldots, l\}$ be the set of indices $i$ such that $o$ is opposite to the given orientation on $D_i$. The homological grading of $\ca(D, o)$ is given by
	\[
    	\gr_h(\ca(D, o)) = 2\sum_{i \in I, j \notin I} \mathit{lk}(D_i, D_j).
    \]
    where $\mathit{lk}$ denotes the linking number. In particular, $\gr_h(\ca(D)) = \gr_h(\cb(D)) = 0$.
\end{proposition}

Thus when $c$ is invertible the graded module structure of $H_{h, t}(D; R)$ is completely known. This is the case for $\QQ$-Lee homology $(c = 2)$ and $\ZZ$-Bar-Natan homology $(c = 1)$. We will see in \Cref{subsec:reduction-of-params} that, even if $c$ is not invertible, the graded module structure of $H_{h, t}(D; R)$ is determined by $c$. 

Finally we state the variances of the Lee classes under the Reidemeister moves and cobordisms. The following \Cref{prop:cab-reidemeister} was the key to defining the link invariant $\bar{s}_c$ in \cite{Sano:2020}, which in particular implies that the Lee classes are \textit{not} invariant under the Reidemeister moves. 
% One objective of the paper was to refine the classes so that it actually becomes invariant under the Reidemeister moves. This paper follows the same path in the reduced setup. 

\begin{definition}
    For any unary function $f$, its \textit{difference function} $\delta f$ is defined by
    \[
        \delta f(x, y) = f(y) - f(x).
    \]
\end{definition}

\begin{definition}
    For an oriented link diagram $D$, let $w(D)$ denote the writhe of $D$ and $r(D)$ denote the number of Seifert circles of $D$.
\end{definition}

\begin{proposition}[{\cite[Proposition 2.13]{Sano:2020}}] \label{prop:cab-reidemeister}
    Suppose $D, D'$ are related by a single Reidemeister move. The corresponding isomorphism $\rho$ maps the Lee classes as
    \begin{equation} \label{eq:cab-reidemeister}
    \begin{aligned}
        \ca(D) &\xmapsto{\ \rho\ } \epsilon c^j \ca(D'), \\
        \cb(D) &\xmapsto{\ \rho\ } \epsilon' c^j \cb(D').
    \end{aligned}
    \end{equation}
    Here $j \in \{-1, 0, 1\}$ given by
    \[ 
        j = \frac{\delta w(D, D') - \delta r(D, D')}{2} 
    \]
    and $\epsilon, \epsilon' \in \{ \pm 1 \}$ are signs satisfying
    \[
        \epsilon \epsilon' = (-1)^j.
    \]
\end{proposition}

\begin{remark}
    Here it is not assumed that $c$ is invertible. When $j < 0$, equations \eqref{eq:cab-reidemeister} should be understood as
    \begin{align*}
        \ca(D') &\xmapsto{\rho^{-1}} \epsilon c^{-j} \ca(D), \\
        \cb(D') &\xmapsto{\rho^{-1}} \epsilon' c^{-j} \cb(D).
    \end{align*}
    %
    % In particular we may consider the case $c = 0$, which implies the invariance of Plamenevskaya's invariant $\psi(L)$ for a transverse link $L$ \cite{Plamenevskaya:2006}. 
\end{remark}

\begin{remark}
    For each orientation $o \in O(D)$ and the corresponding orientation $o' \in O(D')$, the relations given in \Cref{prop:cab-reidemeister} also hold between $\ca(D, o) \in H(D)$ and $\ca(D', o') \in H(D')$ (with $j, \epsilon, \epsilon'$ depending on $(o, o')$). This is because $C(D)$ and $C(D_o)$ only differ by some bigrading shift, and the cycles and the map $\rho$ correspond relevantly. 
\end{remark}

The proof of {\cite[Proposition 2.13]{Sano:2020}} is based on the element-wise description of $\rho$. We give a more simple proof in \Cref{sec:proof-of-key-prop}, based on the diagrammatic description of $\rho$.

\begin{proposition}[{\cite[Proposition 3.17]{Sano:2020}}] \label{prop:cab-cobordism}
    Suppose $(R, h, t)$ is factorable and $c$ is invertible. Let $S$ be an oriented cobordism between links $L, L'$ that has no closed components. Let $D, D'$ be the diagrams of $L, L'$ respectively, and  $\phi$ be the cobordism map corresponding to $S$
    \[
        \phi: H_{h, t}(D; R) \rightarrow H_{h, t}(D'; R).
    \]
    Then $\phi$ maps the Lee classes as 
    \begin{align*}
        \ca(D) &\xmapsto{\ \phi\ } \epsilon c^j \ca(D') + \cdots, \\
        \cb(D) &\xmapsto{\ \phi\ } \epsilon' c^j \cb(D') + \cdots
    \end{align*}
    where $j \in \ZZ$ is given by 
    \[
        j = \frac{\delta w(D, D') - \delta r(D, D') - \chi(S)}{2}
    \]
    and $\epsilon, \epsilon' \in \{ \pm 1 \}$ are signs satisfying
    \[
        \epsilon \epsilon' = (-1)^j.
    \]
    Moreover if every component of $S$ has a boundary in $L$, then the $(\cdots)$ terms vanish.
\end{proposition}

    \subsection{Reduction of parameters}
\label{subsec:reduction-of-params}

Here we continue to assume that $(R, h, t)$ is factorable. It will be convenient to consider another basis $\{1, X_a\}$ for $A_{h, t}$, so that the operations on $A_{h, t}$ are described as 
\begin{equation}
    \label{eq:1Xa-ops}
    \begin{gathered}
        m(1 \otimes 1) = 1, \quad 
        m(X_a \otimes 1) = m(1 \otimes X_a) = X_a, \quad
    	m(X_a \otimes X_a) = c X_a, \\
        \Delta(1) = X_a \otimes 1 + 1 \otimes X_a - c (1 \otimes 1), \quad 
    	\Delta(X_a) = X_a \otimes X_a, \\
    	\iota(1) = 1, \quad \epsilon(1) = 0,\quad \epsilon(X_a) = 1.
    \end{gathered}
\end{equation}

\begin{definition}
    Let $A$ be a Frobenius algebra and $\theta$ an invertible element in $A$. The \textit{twisting} $A_\theta$ of $A$ by $\theta$ is another Frobenius algebra $(A, m, \iota, \Delta_\theta, \epsilon_\theta)$ with the same algebra structure as $A$ but with a different coalgebra structure given by
    \[
        \Delta_\theta(x) = \Delta(\theta^{-1}x),
        \quad 
        \epsilon_\theta(x) = \epsilon(\theta x).
    \]
\end{definition}

\begin{lemma}
\label{lem:frob-alg-isom}
    Suppose $(R, h, t)$, $(R, h', t')$ are both factorable, and that $c = \sqrt{h^2 + 4t}$ and $c' = \sqrt{h'^2 + 4t'}$ are related as $c' = \theta c$ for some invertible $\theta \in R$. Then there is a Frobenius algebra isomorphism
    \[
        \psi: A_{h, t} \longrightarrow A_{h', t'; \theta}
    \]
    mapping
    \[
        X_a \mapsto \theta^{-1}X_{a'},\quad
        X_b \mapsto \theta^{-1}X_{b'}.
    \]
    Here $A_{h', t'; \theta}$ denotes the $\theta$-twisting of $A_{h',t'}$. Moreover these maps satisfy the cocycle condition, i.e.\ the following diagram consisting of the above described maps commute.
	\begin{equation}
        \begin{tikzcd}
            {A_{h, t}} \arrow[rr, "\psi''"] \arrow[rd, "\psi"] &  & {A_{h'', t''; \theta\theta'}} \\
             & {A_{h', t'; \theta}} \arrow[ru, "\psi'"] & 
        \end{tikzcd}
    \end{equation}
\end{lemma} 

\begin{proof}
    Define a ring isomorphism
    \[
        \psi: R[X] \rightarrow R[X],\quad X\ \mapsto\ \theta^{-1} (X - a') + a.
    \]
    Using \eqref{eq:1Xa-ops} it is easy to check that $\psi$ induces the desired Frobenius algebra isomorphism.
\end{proof}

\begin{proposition} \label{prop:ht-relation} 
    Suppose $(R, h, t)$, $(R, h', t')$ satisfy the condition of \Cref{lem:frob-alg-isom}. Then for any link diagram $D$, there is a chain isomorphism 
    \[
        \psi: C_{h, t}(D; R) \longrightarrow C_{h', t'}(D; R).
    \]
    Moreover when $\theta = 1$, the above map is natural with respect to $D$, i.e.\ if two diagrams $D, D'$ are related by a single Reidemeister move, the following diagram commutes.
    \begin{equation*}
        \centering
        \begin{tikzcd}
            C_{h, t}(D) \arrow{r}{\psi} \arrow{d}{\rho} & 
            C_{h', t'}(D) \arrow{d}{\rho} \\
            C_{h, t}(D') \arrow{r}{\psi} & 
            C_{h', t'}(D')
        \end{tikzcd}
    \end{equation*}
    Here $\rho$ is the corresponding R-move map of \Cref{prop:Kh-invariance}.
\end{proposition}

\begin{proof}
    Let $C_{h', t'; \theta}(-; R)$ denote the Khovanov complex corresponding to the Frobenius algebra $A_{h', t'; \theta}$. The Frobenius algebra isomorphism $\psi$ of \Cref{lem:frob-alg-isom} induces a chain isomorphism $C_{h, t}(D; R) \rightarrow C_{h', t'; \theta}(D; R)$. Postcomposing the chain isomorphism $C_{h', t'; \theta}(D; R) \rightarrow C_{h', t'}(D; R)$ corresponding to the $\theta$-twisting (see \cite[Proposition 3]{Khovanov:2004}) gives the desired chain isomorphism $\psi$. Naturality follows from the explicit definition of $\psi$ and $\rho$, together with the identity
    \[
        X \otimes 1 - 1 \otimes X = X_a \otimes 1 - 1 \otimes X_a
    \]
    for the case of R1-move. 
\end{proof}

% \begin{remark} \label{rem:psi-non-natural}
%     When $c \neq 1$, naturality does not hold in general due to the effect of twisting. (TODO) example
% \end{remark}

\begin{corollary}
    If $(R, h, t)$ is factorable, then $C_{h, t}(D; R)$ is isomorphic to $C_{c, 0}(D; R)$. If in addition $c/2 \in R$, it is isomorphic to $C_{0, (c/2)^2}(D; R)$.
\end{corollary}

\begin{corollary} 
    $C_\Kh(-; \FF_2) \isom C_\Lee(-; \FF_2)$.
\end{corollary}

\begin{figure}[t]
    \centering
    \includegraphics[width=0.5\textwidth]{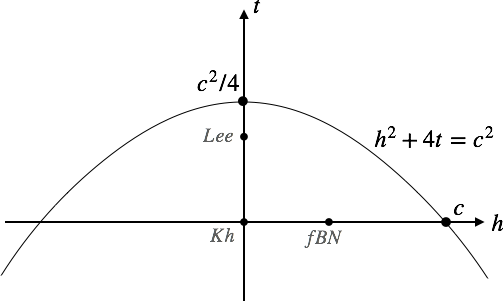}
    \caption{Visualization of the isomorphism class of $C_{h, t}$.}
    \label{fig:ht-correspondence}
\end{figure}

\Cref{prop:ht-relation} implies that $c$ determines the isomorphism class of $C_{h, t}(D; R)$. The isomorphism class can be visualized by the hyperbola $h^2 + 4t = c^2$ on the $ht$-coordinate space as in \Cref{fig:ht-correspondence}. 
% The following proposition implies that the corresponding Lee cycles can be thought as elements of $C_c(D; R)$. 

\begin{proposition} \label{prop:ht-relation-ca}
    Under the assumption of \Cref{prop:ht-relation}, the chain isomorphism $\psi$ maps the Lee cycles of $C_{h, t}(D; R)$ to that of $C_{h', t'}(D; R)$ multiplied by a power of $\theta$. In particular when $\theta = 1$, the Lee cycles of $D$ and $D'$ correspond exactly.
\end{proposition}

\begin{proof}
    $\psi$ maps $X_a$ to $\theta^{-1} X_{a'}$ and $X_b$ to $\theta^{-1} X_{b'}$, and the $\theta$-twisting are given by vertex-wise multiplications of powers of $\theta$. 
\end{proof}

Thus when considering Lee classes and its behavior under Reidemeister moves and cobordisms, it suffices to consider the case $h = c$ and $t = 0$. We occasionally denote $H_c(D)$ for $H_{c, 0}(D)$. The universal theory among all such triples is given by $X^2 - HX = X(X - H)$ over $R = \ZZ[H]$, which is the bigraded Bar-Natan theory, or what is called the \textit{$U(1)$-equivariant theory} (see \cite{Khovanov:2004,Khovanov:2022}). 
    \section{Module structures and reduced homology} \label{sec:module-str}

A module structure on Khovanov homology was first defined in \cite{Khovanov:2002}, and on other variants in \cite{Hedden:2012, Alishahi:2017, Alishahi:2018}. The reduced version of Khovanov homology was also defined by Khovanov in \cite{Khovanov:2002}, and for other variants, the reduced Bar-Natan homology is given by Kotelskiy--Watson--Zibrowius in \cite{KWZ:2019} and the reduced $U(1) \times U(1)$-equivariant Khovanov homology by Akhmechet--Zhang in \cite{AZ:2022}. Here we generalize these structures for a general triple $(R, h, t)$.

\subsection{Module structures}
\label{subsec:module-str}

\begin{definition}
    A \textit{pointed link} $(L, p)$ is a link $L$ with a marked point $p \in L$. A \textit{pointed link diagram} $(D, p)$ is defined likewise, where the point $p$ lies on an arc of $D$.
\end{definition}

Let $(D, p)$ be a pointed link diagram. First we define an endomorphism $x_p$ on $C_{h, t}(D)$ as follows: Take a small circle $\bigcirc$ near $p$. Merging $\bigcirc$ into a neighborhood of $p$ corresponds to the multiplication
\[
    m_p: A_{h,t} \otimes C_{h,t}(D) \rightarrow C_{h,t}(D).
\]
Define
\[
    x_p = m_p(X \otimes -) : C_{h, t}(D) \rightarrow C_{h, t}(D).
\]
The following proposition is a generalization of \cite[Lemma 2.3]{Hedden:2012}, \cite[Lemma 2.1]{Alishahi:2017} and \cite[Lemma 3.3]{Alishahi:2018}.

\begin{proposition} \label{lem:x_p endo}
    Suppose $p, q$ are two marked points on $D$ separated by a crossing $c$. Then $x_p$ and $h - x_q$ are chain homotopic.
\end{proposition}

\begin{proof}
    Let $D_0, D_1$ be the diagrams obtained from $D$ by 0-, 1-resolving the crossing $c$ respectively. There are chain maps between $C(D_0)$ and $C(D_1)$ corresponding to the saddle moves in both ways 
    \[
    \begin{tikzcd}
        C(D_0) \arrow[r, "f", shift left] & C(D_1) \arrow[l, "g", shift left].
    \end{tikzcd}
    \]
    We may view $C(D)$ as the mapping cone of $f$ with differential
    \[
        d  = \begin{pmatrix} -d_0 & 0 \\ f & d_1 \end{pmatrix}.
    \]
    We claim that $x_p + x_q - h$ are null homotopic by the chain homotopy 
    \[
        H = \begin{pmatrix} 0 & g \\ 0 & 0 \end{pmatrix}.
    \]
    First we have 
    \[
        dH + Hd 
        = \begin{pmatrix} g f & -d_0g + g d_1 \\ 0 & f g \end{pmatrix}
        = \begin{pmatrix} g f & 0 \\ 0 & f g \end{pmatrix}.
    \]
    Take any state $u$ and focus on the $u$-summand of $C(D)$. Here we assume $u(c) = 0$ since the proof for the other case is identical. In this case it suffices to prove that 
    \[
        gf = x_p + x_q - h.
    \]
    If the points $p, q$ belong to the same circle of $D(u)$, then $g f$ is a merge-after-split, while $x_p, x_q$ are both multiplication by $X$ on the corresponding tensor multiplicand, so \[
        g f = m \Delta = 2 X - h.
    \]
    Otherwise if $p, q$ belong to different circles of $D(u)$, then $g f$ is a split-after-merge and
    \[
        g f = \Delta m = x_p + x_q - h
    \]
    can be checked directly (or by the \textit{neck-cutting relation} of \cite[Equation 13]{Khovanov:2022}). 
\end{proof}

\begin{remark}
    In particular if $h = 2 = 0$ in $R$, which is the case for $\FF_2$-Khovanov homology, it follows that $x_p$ and $x_q$ are chain homotopic, which is the case proved in \cite{Hedden:2012}. 
\end{remark}

In view of \Cref{lem:x_p endo} we define an endomorphism $U_p$ on $C_{h, t}(D)$ as follows: color the arcs of $D$ according to \Cref{algo:ab-coloring}, and define
\[
    U_p = 
    \begin{cases}
        x_p
            & \text{if $p$ is colored $a$,} \\
        h - x_p
            & \text{if $p$ is colored $b$.}
    \end{cases}
\]
Obviously $U_p$ commutes with the differential $d$, and $U_p^2 = hU_p + t$ holds. Regarding $A_{h, t} = R[U]/(U^2 - hU - t)$, we obtain an $A_{h, t}$-module structure on $C_{h, t}(D)$ and on $H_{h, t}(D)$. \Cref{lem:x_p endo} implies that the module structure on $H_{h, t}(D)$ only depends on the component on which $p$ lies. Moreover,

\begin{proposition} \label{lem:R-move-U-action}
    Suppose $D, D'$ are pointed link diagrams related by a Reidemeister move that does not contain the marked points in the changing disk. Then the corresponding R-move map $\rho$ commutes with $U_p$.
\end{proposition}

% check more precisely.

\begin{proof}
    Consider the diagram $\tilde{D}$ obtained from $D$ by adding a positive twist near $p$. We may regard $C(\tilde{D})$ as the mapping cone of 
    \[
        m_p: A \otimes C(D) \rightarrow C(D).
    \]
    Consider the similar diagram for $D'$, then the corresponding R-move map is given by 
    \[
        \begin{pmatrix}
            1 \otimes \rho & 0 \\
            0 & \rho
        \end{pmatrix}.
    \]
    From the fact that is a chain map, we get 
    \[
        m_p \circ (1 \otimes \rho) = \rho \circ m_p
    \]
    and hence the desired result. (Alternatively, in view of  \cite{BarNatan:2004,Khovanov:2022}, we may regard $x_p$ as a cobordism that merges a dotted cup near $p$. The R-move map $\rho$ is also represented by a cobordism, and the composition of the two cobordisms is obviously commutative.)
\end{proof}

\begin{proposition} 
    Let $D$ be a pointed link diagram. Suppose $(R, h, t), (R, h', t')$ are both factorable with $c = c'$. Then with the isomorphism $\psi$ of  \Cref{prop:ht-relation} the following diagram commutes
    \begin{equation*}
        \centering
        \begin{tikzcd}
            A_{h, t} \otimes C_{h, t}(D) \arrow{r}{m_p} \arrow{d}{\psi \otimes \psi} & 
            C_{h, t}(D) \arrow{d}{\psi} \\
            A_{h', t'} \otimes C_{h', t'}(D) \arrow{r}{m_p} & 
            C_{h', t'}(D)
        \end{tikzcd}
    \end{equation*}
\end{proposition}

\begin{proof}
    $\psi$ is a Frobenius algebra isomorphism.
\end{proof}

Thus we conclude that, for a pointed link $L$, there is an $A_{h, t}$-module structure on $H_{h, t}(L)$ whose isomorphism class only depends on $c$ and on the component on which the marked point $p$ lies.

% \begin{definition}
%     The conjugate action $\bar{U}$ is defined by 
%     \[
%         \bar{U} = h - U.
%     \]
% \end{definition}

% \begin{proposition} \label{prop:U-on-cab}
%     Suppose $(R, h, t)$ is factorable. For a pointed link diagram  $D$, 
%     %
%     \begin{align*}
%         U \cdot \ca(D) &= b \cdot \ca(D),\\
%         U \cdot \cb(D) &= a \cdot \cb(D),\\
%         \bar{U} \cdot \ca(D) &= a \cdot \ca(D),\\
%         \bar{U} \cdot \cb(D) &= b \cdot \cb(D)\ 
%     \end{align*}
%     %
%     and hence
%     %
%     \begin{align*}
%         (U - \bar{U}) \cdot \ca(D) &= c \cdot \ca(D),\\
%         (U - \bar{U}) \cdot \cb(D) &= -c \cdot \cb(D).
%     \end{align*}
% \end{proposition}

% \begin{proof}
%     Immediate from \eqref{eq:ab-operations} with 
%     \[
%         U = 
%         \begin{cases}
%             X_a + a = X_b + b
%                 & \text{if $p$ is colored $a$,} \\
%             a - X_b = b - X_a
%                 & \text{if $p$ is colored $b$}
%         \end{cases}
%     \]
%     and the opposite for $\bar{U}$.
% \end{proof}

% \begin{remark} \label{rem:U-intuition}
%     $U - \bar{U} = 2U - h$ is given by $\pm m \Delta(1)$. In view of \cite{BarNatan:2004,Khovanov:2022} this is represented by the identity cobordism with a torus attached near $p$. Moreover $(U - \bar{U})^2 = 4h^2 + t = c^2$, and this is represented by the identity cobordism with a genus-2 surface attached near $p$. See \cite[Equations (16), (19)]{Khovanov:2022}.
% \end{remark}

    \subsection{Cobordism maps} \label{subsec:module-cobordism}

Here we consider \textit{cobordisms} between pointed links, and define \textit{cobordism maps} between the corresponding chain complexes as $A_{h, t}$-module homomorphisms. First we define an explicit chain homotopy equivalence that represents the change of marked points. We follow the standard argument given in \cite{Khovanov:2002}, which is also used for other variants in \cite{KWZ:2019, AZ:2022}.
% In papers \cite{Khovanov:2002, KWZ:2019, AZ:2022} such move is realized by passing arcs `through the infinity', i.e.\ performing planar isotopies in $S^2$. Here we propose a rather more direct approach. 

\begin{proposition} \label{prop:isotopy-through-infty}
    Suppose $D, D'$ are diagrams (in $\RR^2$) related by an isotopy in $S^2$ that passes an arc `through infinity'. Then there is a chain isomorphism 
    \[
        I: C_{h, t}(D) \rightarrow C_{h, t}(D')
    \]
    such that the Lee cycles of the two diagrams correspond exactly (up to sign). Moreover if $p, p'$ are marked points of $D, D'$ related by the isotopy, then 
    \[
        I \circ U_p = U_{p'} \circ I.
    \]
\end{proposition}

\begin{proof}
    There is an obvious isomorphism
    \[
        I_1: C(D) \rightarrow C(D')
    \]
    that maps the enhanced states identically to the corresponding ones. Consider another chain isomorphism $I_2$ induced from the Frobenius algebra isomorphism, induced from the ring isomorphism
    \[
        R[X] \rightarrow R[X],\quad
        X \mapsto h - X.
    \]
    Note that $I_2$ exchanges $X_a$ and $X_b$ up to sign. Define $I = I_2 \circ I_1$. One sees from \Cref{algo:ab-coloring} that the $ab$-colorings of the Seifert circles of $D$ gets flipped as an arc passes through infinity. This effect is taken care by $I_2$ and hence $I$ commutes with $U$. The latter statement is also obvious since $p$ and $p'$ are colored differently. 
\end{proof}

\begin{proposition} \label{prop:tau_pq}
    Suppose $p, q$ are two marked points on $D$ separated by a crossing. Then there is a self-chain homotopy equivalence $\tau_{p, q}$ on $C_{h, t}(D)$ such that
    \[
        \tau_{p, q} \circ U_p = U_q \circ \tau_{p, q}.
    \]
\end{proposition}

\begin{proof}
    Instead of moving the marked point over or under the strand, the desired move can be realized by a sequence of planer isotopies and Reidemeister moves as follows: pull the strand that separates the two points to the outermost region in $\RR^2$. Pass this arc through infinity, and then slide it back close to the original position so that the resulting diagram is identical to the original one, except that the marked point $p$ is moved to $q$. The desired map $\tau_{p, q}$ is given by the composition of the corresponding R-move maps, with the isomorphism $I$ of \Cref{prop:isotopy-through-infty} placed in between. Commutativity with $U$ follows from \Cref{lem:R-move-U-action,prop:isotopy-through-infty}.
\end{proof}

% claim that the reduced versions of \Cref{prop:cab-reidemeister,prop:cab-cobordism} also hold verbatim. 

\begin{definition}
    A \textit{cobordism} between pointed links $L, L'$ in $\RR^3$ is a compact oriented surface in $\RR^3 \times I$ with boundary $-L \times \{0\} \cup L' \times \{1\}$, together with an embedded curve $\gamma \subset S$ that connects the marked points of $L, L'$.
    % and $\gamma(t) \in S^3 \times \{t\}$ for each $t \in I$.
\end{definition}

% Let $\RR^3_t$ denote $\RR^3 \times \{t\}$ for each $t \in I$ and $p: \RR^3 \rightarrow \RR^2$ be the standard projection. 
Given a cobordism $S$ as above, we may isotope $S$ (rel boundary, together with $\gamma$) so that $\gamma$ intersects each slice $\RR^3 \times \{t\}$ at a single point and that $S$ can be represented by a finite sequence of local moves between pointed link diagrams, each of which is either (i) a Reidemeister move, (ii) a Morse move or (iii) a marked point crossing move. This induces a homomorphism
\[
    \phi: C_{h, t}(D) \rightarrow C_{h, t}(D')
\]
by the composition of corresponding homomorphisms. 

\begin{proposition} \label{prop:module-cob}
    The above constructed $\phi$ is an $A_{h, t}$-module homomorphism. 
\end{proposition}

\begin{proof}
    That $U$ commutes with the R-move maps and $\tau_{p, q}$ is already proved. Analogous result for the Morse moves also holds. 
\end{proof}

\begin{remark}
    Although not necessary in this paper, it is interesting to ask whether $\phi$ is independent (up to sign) of the choice of the isotopy of $S$.
\end{remark}
    \subsection{Reduced homology}

For a pointed link diagram $(D, p)$, let $(X_a)_p C_{h, t}(D)$ denote the subcomplex of $C_{h, t}(D)$ which is generated by the $1X_a$-enhanced states whose tensor factor corresponding to the marked circle is restricted to $X_a$. This is indeed a subcomplex since the operations involving the marked circles are given by 
\[
    \underline{X_a} \cdot 1 = \underline{X_a},\quad 
    \underline{X_a}  \cdot X = b\underline{X_a},\quad
    \Delta \underline{X_a} = \underline{X_a} \otimes X_a.
\]
Here the underline indicates the factor corresponding to the marked circle. We also declare $\deg(\underline{X_a}) = 0$ so that the subcomplex inherits the quantum grading (which does not agree with the one on $C_{h, t}(D)$). The subcomplex $(X_b)_pC_{h, t}(D)$ is defined similarly.

\begin{definition} \label{def:red-C}
    The \textit{reduced Khovanov complex} $\tilde{C}^\pm_{h, t}(D)$ of a pointed link $D$ is defined as follows: color the arcs of $D$ according to \Cref{algo:ab-coloring}, and define 
    \[
        \tilde{C}^+_{h, t}(D) = 
        \begin{cases}
            (X_a)_p C_{h, t}(D)
                & \text{if $p$ is colored $a$,} \\
            (X_b)_p C_{h, t}(D)
                & \text{if $p$ is colored $b$}
        \end{cases}
    \]
    and 
    \[
        \tilde{C}^-_{h, t}(D) = 
        \begin{cases}
            (X_b)_p C_{h, t}(D)
                & \text{if $p$ is colored $a$,} \\
            (X_a)_p C_{h, t}(D)
                & \text{if $p$ is colored $b$.}
        \end{cases}
    \]
    The corresponding homology groups are denoted $\tilde{H}^\pm_{h, t}(D)$. 
\end{definition}

With the endomorphism $U_p$ given in \Cref{subsec:module-str}, we may alternatively define
\[
    \tilde{C}^+_{h, t}(D) = \Ima(U_p - a), \quad \tilde{C}^-_{h, t}(D) = \Ima(U_p - b).
\]

\begin{proposition} \label{prop:red-C-isom1}
    There is a bigrading preserving involution $I$ on $C_{h, t}(D)$ that maps $C^+_{h, t}(D)$ isomorphically onto $C^-_{h, t}(D)$.
\end{proposition}

\begin{proof}
    Let $I$ be the chain isomorphism $\psi$ of \Cref{prop:ht-relation} for the case $\theta = -1$. The statement can be seen from the fact that $I$ is induced from the ring isomorphism $X \mapsto h - X$.
\end{proof}

\begin{proposition} \label{prop:C(D)-SES}
    There is a short exact sequence of complexes
    \[
        0 
            \rightarrow
        \tilde{C}^+_{h, t}(D) 
            \xrightarrow{i} 
        C_{h, t}(D) 
            \xrightarrow{\pi}
        \tilde{C}^-_{h, t}(D)
            \rightarrow 
        0.
    \]
    Moreover this sequence splits when $c$ is invertible in $R$
\end{proposition}

To prove this, we first give alternative descriptions of $C^\pm_{h, t}(D)$. Define a quotient complex
\[
    C_{h, t}(D) / (X_a)_p = C_{h, t}(D) / (X_a)_p C_{h, t}(D).
\]
We may regard $C_{h, t}(D) / (X_a)_p$ as a chain complex generated by the enhanced states whose tensor factor corresponding to the marked circle is $1$. The operations involving the marked circle are given by 
\[
    \underline{1} \cdot 1 = \underline{1},\quad 
    \underline{1} \cdot X = a\underline{1},\quad
    \Delta \underline{1} = \underline{1} \otimes X_b.
\]
We declare $\deg(\underline{1}) = 0$ so that the quotient complex inherits the quantum grading. Similarly define $C_{h, t}(D) / (X_b)_p$.

\begin{proposition} \label{prop:red-C-isom2}
    There are bigraded preserving chain isomorphisms
    \begin{align*}
        (X_a)_p C_{h, t}(D) &\isom C_{h, t}(D) / (X_b)_p,\\
        (X_b)_p C_{h, t}(D) &\isom C_{h, t}(D) / (X_a)_p.
    \end{align*}
\end{proposition}

\begin{proof}
    The first isomorphism is given by
    \[
        \underline{X_a} \otimes x \mapsto \underline{1} \otimes x
    \]
    and similarly for the second. 
\end{proof}

\begin{proof}[Proof of \Cref{prop:C(D)-SES}]
    We assume that $p$ is colored $a$. There is a short exact sequence
    \[
        0 
            \rightarrow
        (X_a)_p C(D)
            \xhookrightarrow{i_a}
        C(D) 
            \xrightarrow{\pi_a}
        C(D)/(X_a)_p \isom (X_b)_p C(D)
            \rightarrow 
        0.
    \]
    The other inclusion 
    \[
        i_b: (X_b)_p C(D) \hookrightarrow C(D)
    \]
    postcomposed with $\pi_a$ is the multiplication by $a - b = -c$. Thus when $c$ is invertible, $-c^{-1} i_b$ gives a splitting.
\end{proof}

\begin{corollary} \label{cor:redH-rank}
    If $c$ is invertible, $\tilde{H}^\pm_{h, t}(D)$ is a free $R$-module of rank $2^{|D| - 1}$. \qed
\end{corollary}

Next we prove that the isomorphism class of the reduced homology is an invariant of pointed links.

\begin{proposition} \label{prop:red-r-moves}
    Suppose $D, D'$ are pointed link diagrams related by a Reidemeister move that does not contain the marked point in the changing disk. Then the corresponding R-move map $\rho$ induces chain homotopy equivalences
    \[
        \rho: \tilde{C}^\pm_{h, t}(D) \rightarrow \tilde{C}^\pm_{h, t}(D').
    \]
\end{proposition}

\begin{proof}
    From the explicit descriptions of the R-move maps the corresponding chain homotopies given in \cite[Section 4.3]{BarNatan:2004} we see that the $ab$-coloring on the marked circles remains unchanged by these maps. 
\end{proof}

\begin{proposition} \label{prop:red-m-moves}
    Suppose $D, D'$ are pointed link diagrams related by a Morse move that does not contain the marked point in the changing disk. Then the corresponding map $\rho$ induces homomorphisms
    \[
        \rho: \tilde{C}^\pm_{h, t}(D) \rightarrow \tilde{C}^\pm_{h, t}(D').
    \]
\end{proposition}

\begin{proof}
    Similar to the proof of \Cref{prop:red-r-moves}.
\end{proof}

\begin{proposition} \label{prop:red-basept-cross}
    Suppose $p, q$ are two marked points on $D$ separated by a crossing. Then the chain homotopy equivalence $\tau_{p, q}$ of \Cref{prop:tau_pq} induces chain homotopy equivalences
    \[
        \tau_{p, q}: \tilde{C}^\pm_{h, t}(D, p) \rightarrow \tilde{C}^\pm_{h, t}(D, q).
    \]
\end{proposition}

\begin{proof}
    That $\tau_{p, q}$ induces homomorphisms between the reduced complexes can be seen by the commutative diagram
    \[
    \begin{tikzcd}
    C(D) \arrow[r, two heads, "U_p - a"] \arrow[d, "{\tau_{p, q}}"] & {\tilde{C}^+(D, p)} \arrow[d, "{\tau_{p, q}}"] \\
    C(D) \arrow[r, two heads, "U_q - a"]           & {\tilde{C}^+(D, q)}                           
    \end{tikzcd}    
    \]
    and similarly for $\tilde{C}^-$. That $\tau_{p, q}$ is a chain homotopy equivalence can be seen from the construction together with \Cref{prop:red-r-moves}.
\end{proof}

\begin{proposition}
    If $(R, h, t), (R, h', t')$ are both factorable with $c = c'$, then the isomorphism $\psi$ of \Cref{prop:ht-relation} induces chain isomorphisms
    \[
        \psi: \tilde{C}^\pm_{h, t}(D) \rightarrow \tilde{C}^\pm_{h', t'}(D').
    \]
\end{proposition}

\begin{proof}
    Immediate from the observation given in the proof of \Cref{prop:ht-relation-ca} with $\theta = 1$.
\end{proof}

Assembling the above obtained results, we conclude that the \textit{reduced Khovanov homology} $\tilde{H}^\pm_{h, t}(L)$ of a pointed link $L$ is well-defined, whose isomorphism class depends only on $c$ and on the component on which the marked point lies. In particular the reduced Khovanov homology of an (unpointed) knot $K$ is well-defined. We also have cobordism maps between the reduced complexes. 

\begin{proposition} \label{prop:red-cobordism}
    Suppose $S$ is a cobordism between pointed links $L, L'$. The cobordism map 
    \[
        \phi: C_{h, t}(D) \rightarrow C_{h, t}(D')
    \]
    given in \Cref{subsec:module-cobordism} restricts to the reduced complexes
    \[
        \phi: \tilde{C}^\pm_{h, t}(D) \rightarrow \tilde{C}^\pm_{h, t}(D').
    \]
\end{proposition}

\begin{proof}
    Immediate from \Cref{prop:module-cob} and \Cref{prop:red-r-moves,prop:red-m-moves,prop:red-basept-cross}.
\end{proof}

In the following we only consider $\tilde{C}^+$ and omit the $+$ symbol when there is no need to consider both complexes at the same time. Finally we remark that the unreduced theory can be recovered from the reduced theory.

\begin{definition}
    Given an unpointed link $L$, a pointed link $L_+$ is defined by adding a disjoint pointed unknot to $L$. For an unpointed link diagram $D$, a pointed link diagram $D_+$ is defined likewise, with the added circle oriented counterclockwise.
\end{definition}

\begin{proposition} \label{prop:red-recover-unred}
    There is an isomorphism
    \[
        C_{h, t}(D) \isom \tilde{C}_{h, t}(D_+)
    \]
    given by the correspondence
    \[
        x \mapsto \underline{X_a} \otimes x.
    \]
    \qed
\end{proposition}

% In order to split the canonical cycles so that one half belongs to $\tilde{C}^+$ and the other half to $\tilde{C}^-$
    \subsection{Reduced Lee classes}

\begin{definition}
    Define $\tilde{O}(D)$ as the subset of $O(D)$ consisting of orientations $o$ whose color at $p$ by the $ab$-coloring of \Cref{algo:ab-coloring} coincides with the one obtained from the given orientation of $D$.
\end{definition}

\begin{definition}
    For each $o \in \tilde{O}(D)$ define the \textit{reduced Lee cycle}
    \[
        \tilde{\ca}(D, o) \in \tilde{C}^+_{h, t}(D)
    \]
    by the preimage of the Lee cycle $\ca(D, o) \in C_{h, t}(D)$ under the inclusion $\tilde{C}^+_{h, t}(D) \hookrightarrow C_{h, t}(D)$. Similarly define
    \[
        \tilde{\cb}(D, o) \in \tilde{C}^-_{h, t}(D)
    \]
    by the preimage of $\cb(D, o) = \ca(D, -o) \in C_{h, t}(D)$ under the inclusion $\tilde{C}^-_{h, t}(D) \hookrightarrow C_{h, t}(D)$. 
\end{definition}

\begin{proposition}
    If $c$ is invertible, then $\tilde{H}^+_{h, t}(D)$ is freely generated by the classes $\tilde{\ca}(D, o)$ and $\tilde{H}^-_{h, t}(D)$ is freely generated by the classes $\tilde{\cb}(D, o)$ over $R$.
\end{proposition}

\begin{proof}
    Immediate from \Cref{prop:ab-gen,prop:C(D)-SES}.
\end{proof}

\begin{corollary}
    If $c$ is invertible and $D$ is a knot diagram, $\tilde{H}_{h, t}(D)$ is freely generated by the single class $\tilde{\ca}(D)$. \qed
\end{corollary}

The following two propositions are reduced versions of \Cref{prop:cab-reidemeister,prop:cab-cobordism}. 

\begin{proposition} \label{prop:cab-reidemeister-red}
     Suppose $D, D'$ are pointed link diagrams related by a Reidemeister move that does not contain the marked point in the changing disk. Then the corresponding R-move map $\rho$ of \Cref{prop:red-r-moves} sends the reduced Lee classes as
    \begin{align*}
        \tilde{\ca}(D) &\xmapsto{\ \rho\ } \epsilon c^j \tilde{\ca}(D'), \\
        \tilde{\cb}(D) &\xmapsto{\ \rho\ } \epsilon' c^j \tilde{\cb}(D').
    \end{align*}
    Here $j \in \{-1, 0, 1\}$ given by
    \[ 
        j = \frac{\delta w(D, D') - \delta r(D, D')}{2} 
    \]
    and $\epsilon, \epsilon' \in \{ \pm 1 \}$ are signs satisfying
    \[
        \epsilon \epsilon' = (-1)^j.
    \]
\end{proposition}

\begin{proof}
    The proof of \Cref{prop:cab-reidemeister} works without change (see \Cref{sec:proof-of-key-prop}).
\end{proof}

\begin{proposition} \label{prop:cab-cobordism-red}
    Suppose $(R, h, t)$ is factorable and $c$ is invertible. Let $S$ be an oriented cobordism between pointed links $L, L'$ that has no closed components. The corresponding cobordism map $\phi$ of \Cref{prop:red-cobordism} sends the reduced Lee classes as 
    \begin{align*}
        \tilde{\ca}(D) &\xmapsto{\ \phi\ } \epsilon c^j \tilde{\ca}(D') + \cdots, \\
        \tilde{\cb}(D) &\xmapsto{\ \phi\ } \epsilon' c^j \tilde{\cb}(D') + \cdots
    \end{align*}
    where $j \in \ZZ$ is given by 
    \[
        j = \frac{\delta w(D, D') - \delta r(D, D') - \chi(S)}{2}
    \]
    and $\epsilon, \epsilon' \in \{ \pm 1 \}$ are signs satisfying
    \[
        \epsilon \epsilon' = (-1)^j.
    \]
    Moreover if every component of $S$ has a boundary in $L$, then the $(\cdots)$ terms vanish.
\end{proposition}

\begin{proof} 
     Again the proof of \Cref{prop:cab-cobordism} works without change, except that we need to consider the effect of $\tau_{p, q}$, which in fact needs no care since it is defined by a composition of R-move maps and the map $I$ of \Cref{prop:isotopy-through-infty}.
\end{proof}

\begin{proposition}
    Under the identification given in  \Cref{prop:red-recover-unred}, the unreduced and the reduced Lee classes correspond as
    \[
        \ca(D, o) = \tilde{\ca}(D_+, o_+)
    \]
    for each $o \in O(D)$. 
    \qed
\end{proposition}

    \subsection{Connected sums and mirrors}

\subsubsection{Connected sums}

\begin{proposition} \label{prop:C-red-conn-sum}
    Suppose $D, D'$ are pointed link diagrams. Then there is a chain isomorphism
    \[
        \tilde{C}_{h, t}(D \# D') \isom \tilde{C}_{h, t}(D) \otimes_R \tilde{C}_{h, t}(D').
    \]
    Under this identification, the Lee cycles correspond as
    \[
        \tilde{\ca}(D \# D', o \# o') = \tilde{\ca}(D, o) \otimes \tilde{\ca}(D', o')
    \]
    for any $o \in \tilde{O}(D)$ and $o' \in \tilde{O}(D')$. Here it is assumed that, both $D, D'$ have the marked points on outermost arcs, and that $D \# D'$ can be realized by a surgery along a coherently oriented untwisted band that connects the two marked points. 
\end{proposition}

\begin{proof}
    From the assumption, the two marked points are colored the same by the $ab$-colorings of $D$ and $D'$. The isomorphism is given by mapping
    \[
        \underline{1} \otimes x \otimes y \mapsto 
        (\underline{1} \otimes x) \otimes (\underline{1} \otimes y).
    \]
    Here we used the quotient description of the reduced complexes.
\end{proof}

\begin{remark} \label{rem:unred-disj-union}
    For unpointed link diagrams $D, D'$, we have $(D \sqcup D')_+ =  D_+ \# D'_+$ and hence obtain the well known formula
    \[
        C_{h, t}(D \sqcup D') \isom C_{h, t}(D) \otimes_R C_{h, t}(D')
    \]
    and 
    \[
        \ca(D \sqcup D') = \ca(D) \otimes \ca(D').
    \]
\end{remark}

\subsubsection{Mirrors}

For any Frobenius algebra $A = (A, m, \iota, \Delta, \epsilon)$ over $R$, its \textit{dual Frobenius algebra} is defined by $A^* = (A^*, \Delta^*, \epsilon^*, m^*, \iota^*)$ where $A^* = \Hom_R(A, R)$ and other maps are the dual maps. For a link diagram $D$, let $D^*$ denote its mirror.

\begin{lemma} \label{lem:A-self-dual}
    There is a Frobenius algebra isomorphism
    \[
        \varphi: A_{h, t} \rightarrow A_{h, t}^*
    \]
    given by 
    \[
        \varphi(1) = X^*, \quad
        \varphi(X) = 1^* + h X^*
    \]
    where $\{1^*, X^*\}$ is the basis of $A_{h, t}^*$ dual to the basis $\{1, X\}$ for $A_{h, t}$. 
\end{lemma}

\begin{proof}
    The desired $\varphi$ is given by the composition of two isomorphisms: first we have $A_{h, t} \isom A_{-h, t}^*$ by the correspondence
    \[
        1 \mapsto X^*,\quad
        X \mapsto 1^*.
    \]
    Next the ring isomorphism
    \[
        R[X] \rightarrow R[X];\quad X \mapsto X + h
    \]
    induces $A_{h, t} \isom A_{-h, t}$ and hence $A^*_{-h, t} \isom A^*_{h, t}$ by
    \[
        1^* \mapsto 1^* + hX^*,\quad
        X^* \mapsto X^*.
    \]
\end{proof}

\begin{proposition} \label{prop:C-mirror}
    The isomorphism $\varphi$ of \Cref{lem:A-self-dual} induces a chain isomorphism
    \[
        C_{h, t}(D^*) \isom C_{h, t}(D)^*.
    \]
    This gives a perfect pairing
    \[
        \langle \cdot, \cdot \rangle: C_{h, t}(D) \otimes C_{h, t}(D^*) \rightarrow R
    \]
    defined by 
    \[
        \langle z, w \rangle = \langle z, \varphi(w) \rangle_{\text{std}}
    \]
    where the right hand side $\langle \cdot, \cdot \rangle_{\text{std}}$ is the standard pairing between $C$ and $C^*$. 
\end{proposition}

\begin{proof}
    The composition of chain isomorphisms
    \[
        C_{h, t}(D^*) \isom C_{-h, t}(D)^* \isom C_{h, t}(D)^*
    \]
    is realized by applying $\varphi$ to the tensor factors.
\end{proof}

Now assume $(R, h, t)$ is factorable. 

\begin{lemma} \label{lem:A-pairing-iota}
    The isomorphism $\varphi$ of \Cref{lem:A-self-dual} maps
    \[
        \varphi(X_a) = 1^* + bX^*,\quad
        \varphi(X_b) = 1^* + aX^*
    \]
    and
    \begin{gather*}
        \langle X_a, \varphi(X_a) \rangle = c,\  
        \langle X_b, \varphi(X_b) \rangle = -c,\\
        \langle X_a, \varphi(X_b) \rangle =
        \langle X_b, \varphi(X_a) \rangle = 0.
    \end{gather*}
    \qed
\end{lemma}

\begin{proposition} \label{prop:C-mirror-pairing}
    The Lee cycles of $D$ and $D^*$ pair as
    \begin{gather*}
        \langle \ca(D), \ca(D^*) \rangle = \epsilon c^r, \ 
        \langle \cb(D), \cb(D^*) \rangle = \epsilon' c^r, \\
        \langle \ca(D), \cb(D^*) \rangle = \langle \cb(D), \ca(D^*) \rangle = 0
    \end{gather*}
    where $r = r(D)$ and $\epsilon, \epsilon' \in \{ \pm 1 \}$ are signs such that $\epsilon\epsilon' = (-1)^r$. 
\end{proposition}

\begin{proof}
    Immediate from \Cref{lem:A-pairing-iota} together with the observation that the Seifert circles of $D$ and $D^*$ are identical.
\end{proof}

Next we relate the reduced complex of $D^*$ with the dual of the reduced complex of $D$.

\begin{proposition} \label{prop:tilde-iota}
    The isomorphism $\varphi$ of \Cref{prop:C-mirror} induces isomorphisms $\tilde{\varphi}$ such that the following diagram commutes
    \[
    \begin{tikzcd}[row sep=2.5em]
    (X_a)_pC(D^*) \arrow[r, "\sim"] \arrow[d, "\tilde{\varphi}"] & C(D^*)/(X_b)_p \arrow[d, "\tilde{\varphi}"] \\
    (C(D)/(X_b)_p)^* \arrow[r, "\sim"]                 & ((X_a)_pC(D))^*.                  
    \end{tikzcd}
    \]
    Here the horizontal arrows are the isomorphisms of \Cref{prop:red-C-isom2}. Similar isomorphisms with $a, b$ exchanged also exist. Thus we get isomorphisms
    \[
        \tilde{C}^\pm_{h, t}(D^*) \isom (\tilde{C}^\pm_{h, t}(D))^*.
    \]
\end{proposition}

\begin{proof}
    The desired maps are obtained from the following diagram
    \[
    \begin{tikzcd}[row sep=2.5em]
    (X_a)_p C(D^*) \arrow[r, "i_a"] \arrow[d, "\tilde{\varphi}", dashed] & C(D^*) \arrow[r, "\pi_b"] \arrow[d, "\varphi"] & C(D^*)/(X_b)_p  \arrow[d, "\tilde{\varphi}", dashed] \\
    (C(D)/(X_b)_p)^* \arrow[r, "\pi_b^*"] & C(D)^* \arrow[r, "i_a^*"] & ((X_a)_p C(D))^*
    \end{tikzcd}
    \]
    The unique existence of the dashed arrows follows from \Cref{lem:A-pairing-iota}, for $\varphi(X_a)$ annihilates $X_b$. One can check that the correspondences are given by 
    \[
    \begin{tikzcd}
    \underline{X_a} \otimes x \arrow[d, maps to] \arrow[r, maps to] & \underline{1} \otimes x \arrow[d, maps to] \\
    \underline{1^*} \otimes \varphi(x) \arrow[r, maps to]             & \underline{X_a^*} \otimes \varphi(x)        
    \end{tikzcd}
    \]
    using $\langle 1, \varphi(X_a) \rangle = \langle X_a, \varphi(1) \rangle = 1$. 
\end{proof}

\begin{proposition}
    There are perfect pairings
    \[
        \langle \cdot, \cdot \rangle^\sim: \tilde{C}^\pm_{h, t}(D) \otimes \tilde{C}^\pm_{h, t}(D^*) \rightarrow R
    \]
    such that the following diagrams commute up to sign:
    \[
    \begin{tikzcd}
    \tilde{C}^\pm(D) \otimes \tilde{C}^\pm(D^*) \arrow[r, hook] \arrow[d, "{\langle \cdot, \cdot \rangle^\sim}"] & C(D) \otimes C(D^*) \arrow[d, "{\langle \cdot, \cdot \rangle}"] \\
    R \arrow[r, "c"]                                                                                             & R.                                                            
    \end{tikzcd}
    \]
\end{proposition}

\begin{proof}
    Unraveling the definition of $\tilde{\varphi}$ we get the following commutative diagram
    \[
    \begin{tikzcd}[column sep=4em]
    (X_a)_p C(D) \otimes (X_a)_p C(D^*) \arrow[r, "i_a \otimes i_a", hook] \arrow[d, "1 \otimes \tilde{\varphi}"] & C(D) \otimes C(D^*) \arrow[d, "1 \otimes \varphi"]                            \\
    (X_a)_p C(D) \otimes (C(D)/(X_b)_p)^* \arrow[r, hook, "i_a \otimes \pi_b^*"] \arrow[d, "\sim"]                    & C(D) \otimes C(D)^* \arrow[dd, "{\langle \cdot, \cdot \rangle_\text{std}}"] \\
    (X_a)_p C(D) \otimes ((X_a)_p C(D))^* \arrow[d, "{\langle \cdot, \cdot \rangle_\text{std}}"]                & \\
    R \arrow[r, "c"] & R
    \end{tikzcd}
    \]
    There is a similar commutative diagram with $a, b$ exchanged and the bottom horizontal arrow replaced with $-c$. The composition of the right vertical maps gives $\langle \cdot, \cdot \rangle$. We define $\langle \cdot, \cdot \rangle^\sim$ to be the composition of the left vertical maps.
\end{proof}

\begin{proposition} \label{prop:red-ca-pairing}
    The reduced Lee cycles of $D$ and $D^*$ pair as
    \begin{gather*}
        \langle \tilde{\ca}(D), \tilde{\ca}(D^*) \rangle^\sim = \epsilon c^{r - 1}, \quad
        \langle \tilde{\cb}(D), \tilde{\cb}(D^*) \rangle^\sim = \epsilon' c^{r - 1}
    \end{gather*}
    where $r = r(D)$ and $\epsilon, \epsilon' \in \{ \pm 1 \}$ are signs such that $\epsilon\epsilon' = (-1)^{r - 1}$. \qed
\end{proposition}
    \section{The invariant $\ssr_c$} \label{sec:divisibility}

Now we are ready to define the link invariant $\ssr_c$ from the $c$-divisibility of the reduced Lee class.
% This is the reduced counterpart of the invariant $\bar{s}_c$ given in \cite{Sano:2020}.

\subsection{Setup}

\begin{definition} \label{def:c-divisibility}
	Let $R$ be an integral domain, $M$ an $R$-module, and $c$ a non-zero, non-invertible element in $R$. Define the \textit{$c$-divisibility} of an element $z$ in $M$ by
    \[
        d_c(z) = \max \{\ k \geq 0 \mid z \in c^k M \ \}.
    \]
    We say $z$ is \textit{$c^k$-divisible} if $k \leq d_c(z)$.
\end{definition}

The following lemmas will be used in the coming sections. It is assumed that the assumptions of \Cref{def:c-divisibility} remain valid.

\begin{lemma} \label{lem:c-add}
    For any $z \in M$ and $n \geq 0$,
    \[
        d_c(c^n z) \geq n + d_c(z).
    \]
    Moreover if $M$ has no $c$-torsions, the equality holds. 
\end{lemma}

\begin{proof}
    The inequality is obvious. Suppose $M$ has no $c$-torsion. Put $c^n z = c^k z'$ with $k = d_c(c^n z)$ and some $z' \in M$. From $k \geq n$ we have $c^n(z - c^{k - n}z') = 0$ and hence $z = c^{k - n} z'$ from the assumption. Thus $d_c(z) \geq k - n$.
\end{proof}

\begin{lemma} \label{lem:c-add-2}
    For any $a \in R$ and $z \in M$,
    \[
        d_c(a z) \geq d_c(a) + d_c(z).
    \]
    Moreover if $M$ is free and $c$ is prime, the equality holds.
\end{lemma}

\begin{proof}
    The inequality is obvious. For the latter statement, we may assume $M = R^n$ and put $z = (z_i)$. From the previous lemma we may also assume $d_c(a) = d_c(z) = 0$. If $az$ is $c$-divisible, then $a z_i$ is so for all $i$. Since $c$ is prime, it follows that either $a$ is $c$-divisible or otherwise all $z_i$ are $c$-divisible. Both contradict the assumption and hence $d_c(az) = 0$. 
\end{proof}

\begin{lemma} \label{lem:c-tensor}
    Suppose $M'$ is another $R$-module. For any $z \in M$ and $w \in M'$, 
    \[
        d_c(z \otimes w) \geq d_c(z) + d_c(w).
    \]
    Moreover if $M, M'$ are free and $c$ is prime in $R$, the equality holds.
\end{lemma}

\begin{proof}
    The inequality is obvious. For the latter statement, we may assume $M = R^m$, $M' = R^n$ and identify $M \otimes M'$ with $R^{mn}$. We may also assume $d_c(z) = d_c(w) = 0$. If $z \otimes w$ is $c$-divisible, then $z_i w \in R^n$ is $c$-divisible for each $i$. Then from \Cref{lem:c-add-2}, either $w$ is $c$-divisible or otherwise all $z_i$ are $c$-divisble. Both contradict the assumption and hence $d_c(z \otimes w) = 0$
\end{proof}

\begin{lemma} \label{lem:c-div}
	Let $R'$ be another ring and $M'$ be an $R'$-module. Suppose there is a ring homomorphism $f: R \rightarrow R'$ and an $R$-module homomorphism $\phi: M \rightarrow f^* M'$. Then for any $z \in M$,
	\[
	    d_c(z) \leq d_{f(c)}(\phi(z)).
	\]
	Moreover if $f, \phi$ are isomorphisms, the equality holds. 
\end{lemma}

\begin{proof}
    If $z = c^k z'$ , then $\phi(z) = f(c)^k \phi(z')$ and hence 
    \[
        d_{f(c)}(\phi(z)) \geq  k + d_{f(c)}(\phi(z')) \geq k.
    \]
\end{proof}

\begin{lemma} \label{lem:c-localize}
    Let $c^{-1}R$ denote the localization of $R$ away from $c$ (i.e.\ the minimal extension of $R$ such that $c$ is invertible), and $c^{-1}M$ the module $M \otimes_R c^{-1}R$. If $M$ has no $c$-torsions and two elements $z, w \in M$ are related as $z \otimes 1 = c^n (w \otimes 1)$ in  $c^{-1}M$ for some $n \in \ZZ$, then $d_c(z) = n + d_c(w)$.
\end{lemma}

\begin{proof}
    The natural map $M \rightarrow c^{-1}M$ is injective. If $n \geq 0$ then $z = c^n w$ in $M$, otherwise if $n < 0$ then $c^{-n} z = w$. 
\end{proof}

\begin{lemma} \label{lem:c-localize-at}
    Let $R_{(c)}$ denote the localization of $R$ at $c$ (i.e.\ the minimal extension of $R$ such that all elements in $R \setminus (c)$ are invertible), and $M_{(c)}$ the module $M \otimes_R R_{(c)}$. For any $z \in M$, we have 
    \[
        d_c(z) \leq d_c(z/1).
    \]
    % Here the left-hand side is the divisibility in $M$ and the right-hand side in $M_{(c)}$. 
    Moreover if $M$ is free and $c$ is prime, the equality holds.
    % the $c$-divisibility of $z$ in $M$ is equal to the $c$-divisibility of $z/1$ in $M_{(c)}$.
\end{lemma}

\begin{proof}
    The inequality is obvious from the existence of the natural map
    \[
        M \rightarrow M_{(c)}, \quad z \mapsto z/1.
    \]
    For the latter statement, put $z/1 = c^k (w/s)$ for some $w \in M$ and $s \in R \setminus (c)$ with maximal $k$. Then $sz = c^k w$ in $M$. We have $d_c(s) = 0$ and also $d_c(w) = 0$ from the maximality of $k$. Thus from \Cref{lem:c-add-2} we have $d_c(z) = k$.
\end{proof}

    \subsection{Divisibility of reduced Lee class}
\label{subsec:red-divisibility}

For the remainder of this section, we assume $R$ is an integral domain, $(R, h, t)$ is factorable, and $c$ is non-zero, non-invertible in $R$.

\begin{definition}
    Let $D$ be a pointed link diagram. Define $\tilde{d}_c(D)$ by the $c$-divisibility of the reduced Lee class $\tilde{\ca}(D)$ in $\tilde{H}^+_{h, t}(D)/\Tor{}$, 
    \[
        \tilde{d}_c(D) = d_c(\tilde{\ca}(D)).
    \]
\end{definition}

\begin{example}
    $D = \bigcirc$ has $\tilde{C}(D) = \underline{R}$ and $\tilde{\ca}(D) = \underline{1}$ hence $\tilde{d}_c(D) = 0$.
\end{example}

\begin{example}
    Consider the unknot diagram $D$ with one negative crossing. Suppose $p$ lies on an arc colored $a$ with respect to the given orientation. Then
    \[
        \tilde{C}(D) = \{\ \underline{R} \overset{\Delta}{\longrightarrow} \underline{R} \otimes A \ \}
    \]
    and $\tilde{\ca}(D) = \underline{1} \otimes X_b$. From $\Delta \underline{1} = \underline{1} \otimes X_a$,
    \[
        \tilde{\ca}(D) \sim \underline{1} \otimes (X_b - X_a) = -c (\underline{1} \otimes 1).
    \]
    Since $[\underline{1} \otimes 1]$ generates $\tilde{H}(D) \isom R$, we have $\tilde{d}_c(D) = 1$.
\end{example}

The above examples show that $\tilde{d}_c(D)$ is \textit{not} a pointed link invariant. 

\begin{proposition}
    $\tilde{d}_c(D) =  \tilde{d}_c(-D)$.
\end{proposition}

\begin{proof}
    The involution $I$ of \Cref{prop:red-C-isom1} sends $\tilde{\ca}(D)$ in $\tilde{C}^+(D)$ to $\tilde{\cb}(D) = \tilde{\ca}(-D)$ in $\tilde{C}^-(D) = \tilde{C}^+(-D)$. Thus the result follows from \Cref{lem:c-div}.
\end{proof}

\begin{proposition}
    $\tilde{d}_c(D \sqcup \bigcirc) = \tilde{d}_c(D)$.
\end{proposition}

\begin{proof}
    Suppose $\bigcirc$ is oriented counterclockwise. Then $\tilde{\ca}(D \sqcup \bigcirc) = \tilde{\ca}(D) \otimes X_a$. We have maps in both directions
    \[
    \begin{tikzcd}
    \tilde{C}(D) \arrow[r, "\id \otimes X_a", shift left] & \tilde{C}(D \sqcup \bigcirc) \arrow[l, "\id \otimes \epsilon", shift left]
    \end{tikzcd}
    \]
    such that
    \[
        \tilde{\ca}(D) \longleftrightarrow \tilde{\ca}(D \sqcup \bigcirc).
    \]
    Thus the result follows from \Cref{lem:c-div}.
\end{proof}

\begin{proposition} \label{prop:d_posD}
	If $D$ is a positive diagram, then $\tilde{d}_c(D) = 0$.
\end{proposition}

\begin{proof}
    The orientation preserving state of $D$ is $s = (0, \ldots, 0)$. By 0-resolving the crossings one by one, we get a sequence of quotient maps
    \begin{align*}
    	\tilde{C}(D) \rightarrow \tilde{C}(D_0) \rightarrow \cdots \rightarrow \tilde{C}(D_{0 \cdots 0}).
    \end{align*}
    Since the rightmost diagram is a disjoint union of circles, we have 
    \[
        0 \leq \tilde{d}_c(D) \leq \tilde{d}_c(D_{0 \cdots 0}) = 0.
    \]
\end{proof}

\begin{proposition} \label{prop:d_c-fusion}
    Suppose $D, D'$ are pointed diagrams that are related by a saddle move that splits a Seifert circle of $D$ into two Seifert circles of $D'$ and that does not contain the marked points of $D$, $D'$. Then 
    \[
        \tilde{d}_c(D) \leq \tilde{d}_c(D') \leq \tilde{d}_c(D) + 1.
    \]
\end{proposition}

\begin{proof}
    The sequence of two saddle moves
    \[
        D \rightarrow D' \rightarrow D
    \]
    induces a sequence of chain maps that send the Lee cycles as
    \[
        \tilde{\ca}(D) 
        \ \xmapsto{\Delta} \  
        \tilde{\ca}(D')
        \ \xmapsto{m} \  
        \pm c \cdot \tilde{\ca} (D).
    \]
    Thus the result follows from \Cref{lem:c-add,lem:c-div}.
\end{proof}

\begin{proposition} \label{prop:d_c-conn-sum}
    For pointed link diagrams $D, D'$, 
    \[
        \tilde{d}_c(D \# D') \geq \tilde{d}_c(D) + \tilde{d}_c(D').
    \]
    Moreover when $R$ is a PID and $c$ is prime, this becomes an equality. 
\end{proposition}

\begin{proof}
    From \Cref{prop:C-red-conn-sum} there is a natural map 
    \[
        \tilde{H}(D)/\Tor{} \otimes \tilde{H}(D')/\Tor{} \rightarrow \tilde{H}(D \# D')/\Tor{}
    \]
    that maps
    \[
        [\tilde{\ca}(D)] \otimes [\tilde{\ca}(D')] \mapsto [\tilde{\ca}(D \# D')].
    \]
    Thus the inequality holds from \Cref{lem:c-tensor,lem:c-div}. When $R$ is a PID, this map is an isomorphism and hence the reverse inequality follows from the latter statement of \Cref{lem:c-tensor}. 
\end{proof}

\begin{proposition} \label{prop:d-conn-sum-disj-union}
    For pointed link diagrams $D, D'$,
    \[
        \tilde{d}_c(D \# D') \leq \tilde{d}_c(D \sqcup D') \leq \tilde{d}_c(D \# D') + 1.
    \]
    Here $D \sqcup D'$ is regarded as a pointed link by taking any point on its arc. 
\end{proposition}

\begin{proof}
    This is a special case of \Cref{prop:d_c-fusion}.
\end{proof}

\begin{proposition} \label{prop:RM-k-relation}
    If $D, D'$ are related by a Reidemeister move, then 
    \[
        \tilde{d}_c(D) = \tilde{d}_c(D') + j
    \]
    where 
    \[
        j = \frac{\delta w(D, D') - \delta r(D, D')}{2}.
    \]
\end{proposition}

\begin{proof}
    Immediate from \Cref{prop:cab-reidemeister-red}.
\end{proof}

\begin{proposition} \label{prop:k-ptd-cobordism}
    Suppose $S$ is a cobordism between pointed links $L, L'$, such that each component of $S$ has a boundary in $L$. Then
    \[
        \tilde{d}_c(D) \leq \tilde{d}_c(D') + j
    \]
    where 
    \[
        j = \frac{\delta w(D, D') - \delta r(D, D') - \chi(S)}{2}.
    \]
\end{proposition} 

\begin{proof}
    Immediate from \Cref{prop:cab-cobordism-red} with \Cref{lem:c-localize}.
\end{proof}

    \subsection{Definition and properties of $\ssr_c$}

The following definition is justified from \Cref{prop:RM-k-relation}.

\begin{definition} \label{def:red_sc}
    For any pointed link $L$, define
    \[
        \ssr_c(L) = 2\tilde{d}_c(D) + w(D) - r(D) + 1
    \]
    where $D$ is any pointed diagram representing $L$. 
\end{definition}

\begin{proposition} \label{prop:s-mod-2}
    $\ssr_c(L) \equiv |L| - 1 \bmod{2}$. 
\end{proposition}

\begin{proof}
    Take any diagram $D$ of $L$, and let $S$ be the Seifert surface of $L$ obtained by applying Seifert's algorithm to $D$. Then from 
    \[
        \chi(S) = 2 - 2g(S) - |L| = r(D) - n(D),
    \]
    we have 
    \[
        \ssr_c(L) \equiv n(D) + r(D) + 1 \equiv |L| + 1 \bmod{2}.
    \]  
\end{proof}

The following properties of $\ssr_c$ immediately follows from those of $\tilde{d}_c$. 

\begin{proposition} \label{prop:s-properties}
    \ 
    \begin{enumerate}
        \item $\ssr_c(\bigcirc) = 0$.
        \item $\ssr_c(L \sqcup \bigcirc) = \ssr_c(L) - 1$.
        \item $\ssr_c(L) = \ssr_c(-L)$.
        \item $\ssr_c(L \# L') \geq \ssr_c(L) + \ssr_c(L')$.
        \item $\ssr_c(L \# L') - 1 \leq \ssr_c(L \sqcup L') \leq \ssr_c(L \# L') + 1$.
    \end{enumerate}
    When $R$ is a PID and $c$ is prime in $R$, \textit{4.} becomes an equality. 
    \qed
\end{proposition}

The following proposition states the behavior of $\ssr_c$ under cobordisms, which is the key to proving the main theorem.

\begin{proposition} \label{prop:s-cobordism}
    Suppose $S$ is a cobordism between pointed links $L, L'$, such that each component of $S$ has a boundary in $L$. Then
    \[
        \ssr_c(L) \leq \ssr_c(L') - \chi(S).
    \]
    Moreover, if every component $S$ has boundary in both $L$ and $L'$, then
    \[
        |\ssr_c(L) - \ssr_c(L')| \leq -\chi(S).
    \]
\end{proposition} 

\begin{proof}
    Immediate from \Cref{prop:k-ptd-cobordism}.
\end{proof}

\begin{corollary} \label{cor:s-link-conc-inv}
    $\ssr_c$ is invariant under link concordance. \qed
\end{corollary}

\begin{corollary} \label{cor:s_c-slice}
    If $K$ is a slice knot, then $\ssr_c(K) = 0$. \qed
\end{corollary}

\begin{corollary} \label{thm:s-lower-bound-slice-genus} 
    For a knot $K$, we have 
    \[
        |\ssr_c(K)| \leq 2g_4(K).
    \]
    Here $g_4(K)$ denotes the slice genus of $K$. 
    \qed
\end{corollary}

\begin{corollary}
    For a positive knot $K$,
    \[
        \ssr_c(K) = 2g_4(K) = 2g(K).
    \]
\end{corollary}

\begin{proof}
    Take any positive diagram $D$ of $K$, and let $S$ be a Seifert surface of $K$ obtained by applying Seifert's algorithm to $D$. Then 
    \[
        \ssr_c(K) = n(D) - r(D) + 1 = 2g(S)
    \]
    and
    \[
        2g_4(K) \leq 2g(K) \leq 2g(S) = \ssr_c(K) \leq 2g_4(K).
    \]
\end{proof}

Thus we conclude,

\begin{theorem} \label{thm:1}
    For any PID $R$ and a prime $c$ in $R$, the invariant $\ssr_c$ is a slice-torus invariant. \qed
\end{theorem}

The following property is what all slice-torus invariants have in common.

\begin{corollary}[{\cite[Corollary 5.9]{Lewark:2014}}]
    Suppose $R$ is a PID and $c$ is prime in $R$. For any alternating knot $K$, the invariant $\ssr_c(K)$ coincides with the knot signature $\sigma(K)$ of $K$. 
    \qed
\end{corollary}

% We prove a few more results that follows from \Cref{prop:k-cobordism}. The following is a generalization of \Cref{prop:s_c-mirror}.

% \begin{proposition}[{\cite[Corollary 3]{Livingston:2004}}]
%     Suppose $R$ is a PID and $c$ is prime in $R$. If $K^+$ and $K^-$ differ by a single crossing change, from positive to negative, then
%     \[
%         \ssr_c(K^-) \leq \ssr_c(K^+) \leq \ssr_c(K^-) + 2.
%     \]
%     \qed
% \end{proposition}

% \begin{proposition}
%     For an $\ell$-component pointed link $L$,
%     \[
%         -2(\ell - 1) \leq \ssr_c(L \# {-L^*}) \leq 0.
%     \]
%     For an $\ell$-component pointed link diagram $D$,
%     \[
%         r(D) - \ell \leq \tilde{d}_c(D \# {-D^*}) \leq r(D) - 1.
%     \]
% \end{proposition}

% \begin{proof}
%     First consider the disjoint union $L \sqcup -{L^*}$. By connecting the two pointed components by a band we get $L \# {-L^*}$. Connecting the remaining $\ell - 1$ pairs gives the $\ell$-component unlink, which is represented by a cobordism $S$ of $\chi(S) = -(\ell - 1)$. 
% \end{proof}

% \begin{proposition} \label{prop:s-hat-skein}
%     If $L^+$ and $L^-$ differ by a single crossing change, from positive to negative, then
%     \[
%         \ssr_c(L^-) - 2 \leq \ssr_c(L^+) \leq \ssr_c(L^-) + 2.
%     \]
%     For the corresponding knot diagrams $D^+, D^-$,
%     \[
%         \tilde{d}_c(D^-) - 1 \leq \tilde{d}_c(D^+) \leq \tilde{d}_c(D^-) + 1.
%     \]
% \end{proposition}

    \subsection{Refined Lee classes} 
\label{subsec:refined-class-red}

Here we prove that the Lee classes (modulo torsion) can be refined so that it becomes invariant (up to sign) under the Reidemeister moves.

\begin{definition}
    Let $D$ be a pointed link diagram. Define the \textit{refined Lee class} $\tilde{\cz}(D)$ in $\tilde{H}_{h, t}(D)/\Tor{}$ by
    \[
        \tilde{\cz}(D) = c^{-k} \tilde{\ca}(D)
    \]
    where $k = \tilde{d}_c(D)$. 
\end{definition}

\begin{proposition} \label{prop:cz-reidemeister}
    The class $\tilde{\cz}(D)$ is invariant (up to sign) under the Reidemeister moves.
\end{proposition}

\begin{proof}
    Immediate from \Cref{prop:cab-reidemeister-red}. 
\end{proof}

\begin{proposition} \label{prop:k-cobordism}
    Let $S$ be a cobordism between pointed links $L, L'$, such that each component of $S$ has a boundary in $L$. Then the corresponding cobordism map
    \[
        \phi: \tilde{H}_{h, t}(D)/\Tor \rightarrow \tilde{H}_{h, t}(D')/\Tor
    \]
    sends
    \[
        \tilde{\cz}(D) \xmapsto{\phi} \pm c^l \tilde{\cz}(D')
    \]
    where
    \[
        l = \frac{\delta \tilde s_c(L, L') - \chi(S)}{2}.
    \]
    % $\delta \tilde s_c = $. 
\end{proposition} 

\begin{proof}
    Immediate from \Cref{prop:cab-cobordism-red}.
\end{proof}

\begin{corollary}
    The class $\tilde{\cz}(D)$ is invariant (up to sign) under link concordance. \qed
\end{corollary}

\begin{remark}
    As in \cite{Sano:2020-b} we may also adjust the signs of the cobordism maps so that the sign indeterminacy is removed. Then $\tilde{\cz}(D)$ becomes strictly invariant. 
\end{remark}

% Using the refined Lee class, the invariant $\ssr_c(L)$ for a knot $K$ can be characterized by a cobordism from the unknot $U$ to $K$. 

As a byproduct of the above arguments, we get the following characterization of $\ssr_c(L)$ via cobordism, which is analogous to the reformulation of $s$ given in \cite[Section 2.2]{KM:2011}.

\begin{proposition} \label{prop:s_c-by-cobordism}
    Let $L$ be a pointed link, and $S$ a connected cobordism from the pointed unknot $U$ to $L$. Let $\phi_S$ be the induced cobordism map
    \[
        \phi_S: R =\tilde{H}_{h, t}(U) \rightarrow \tilde{H}_{h, t}(L)/\Tor.
    \]
    Put
    \[
        z = \phi_S(1) \in \tilde{H}_{h, t}(L)/\Tor.
    \]
    Then
    \[
        \ssr_c(L) = 2d_c(z) + \chi(S).
    \]
\end{proposition}

\begin{proof}
    Obviously $\tilde{\ca}(U) = \tilde{\cz}(U) = 1$. From \Cref{prop:k-cobordism},
    \[
        z = \pm c^l \tilde{\cz}(L)
    \]
    where 
    \[
        l = \frac{\ssr_c(L) - \chi(S)}{2}.
    \]
    Now the result is follows from $d_c(\tilde{\cz}(L)) = 0$. 
\end{proof}

Now we restrict to the case when $D$ is a knot diagram.

\begin{proposition} \label{prop:red-H-canon}
    Suppose $R$ is a PID and $c$ is prime. For a knot diagram $D$, the refined Lee class $\tilde{\cz}(D)$ is a generator of $\tilde{H}_{h, t}(D)/\Tor{}$. 
\end{proposition}

\begin{proof}
    From \Cref{prop:red-ca-pairing} we have
    \[
        \langle \tilde{\ca}(D), \tilde{\ca}(D^*) \rangle^\sim 
        = \pm c^{r(D) - 1}.
    \]
    From \Cref{thm:1} it follows that 
    \[
        \tilde{d}_c(D) + \tilde{d}_c(D^*) = r(D) - 1
    \]
    and hence
    \[
        \langle \tilde{\cz}(D), \tilde{\cz}(D^*) \rangle^\sim = \pm 1. 
    \]
    Since $\langle \cdot, \cdot \rangle^\sim$ is a perfect pairing, the classes $\tilde{\cz}(D), \tilde{\cz}(D^*)$ must be generators of $\tilde{H}_{h, t}(D)/\Tor{}$ and $\tilde{H}_{h, t}(D^*)/\Tor{}$ respectively.
\end{proof}

% As in \cite{Sano:2020-b} we may adjust the signs of the R-move maps so that the sign indeterminacy is removed. Then $\tilde{\cz}(D)$ becomes strictly invariant. 
% In particular when $K$ is a knot (and when the assumption of \Cref{prop:red-H-canon} holds) it is justified to call $\tilde{\cz}(K)$ the \textit{canonical generator} of $\tilde{H}_{h, t}(K)/\Tor{}$. 
The following proposition proves the half of \Cref{mainthm:2}.

\begin{proposition}
\label{prop:ssr-rasmussen}
    For $(R, c) = (F[H], H)$ where $F$ is a field of any characteristic, we have $\ssr_H = s^F$ as knot invariants. 
\end{proposition}

\begin{proof}
    Let $\tilde{H}_\BN$ denote the reduced bigraded Bar-Natan homology over $F$, given the triple $(R, h, t) = (F[H], H, 0)$ with $\deg{H} = -2$. From \cite[Proposition 3.8]{KWZ:2019}, $s^F(K)$ is given by the $q$-grading of the generator of $\tilde{H}_\BN(K)/\Tor{}$. Thus with \Cref{prop:red-H-canon} we get
    \[
        s^F(K) = \gr_q(\tilde{\cz}(K)).
    \]
    On the other hand, for any diagram $D$ of $K$, we have
    \[
        \gr_q(\tilde{\ca}(D)) = w(D) - r(D) + 1
    \]
    and hence 
    \[
        \gr_q(\tilde{\cz}(D)) = 2\tilde{d}_H(D) + w(D) - r(D) + 1
    \]
    which is precisely the definition of $\ssr_H(K)$.
\end{proof}

% \begin{question}
%     For a link diagram $D$ we can define classes $\tilde{\cz}(D, o)$ for any orientation $o \in \tilde{O}(D)$ and similarly prove that they are invariant (up to sign) under the Reidemeister moves. Do these classes form a basis for $\tilde{H}_{h, t}(D)/\Tor{}$?
% \end{question}
    \subsection{Classification of $\ssr_c$}
\label{sec:classification}

Here we classify the invariants $\ssr_c$ under the assumption that $R$ is a PID and $c \in R$ is prime. First we divide $(R, c)$ into four types by the following mutually exclusive conditions:

\begin{itemize}
    \item [(A)] $\fchar{R} \neq 0$.
    \item [(B)] $\fchar{R} = 0$ and $c \nmid n$ for every $n \in \ZZ \setminus 0$.
    \item [(C)] $\fchar{R} = 0$ and $c \sim p \cdot 1_R$ for some prime $p \in \ZZ$. 
    \item [(D)] $\fchar{R} = 0$, $c \mid p \cdot 1_R$ and $c \not\sim p \cdot 1_R$ for some prime $p \in \ZZ$. 
\end{itemize}

Here $\sim$ denotes associatedness. Typical examples of $(R, c)$ belonging to the four types are:

\begin{itemize}
    \item [(A)] $(R, c) = (\FF_p[h], h)$, $p$: prime.
    \item [(B)] $(R, c) = (\QQ[h], h)$.
    \item [(C)] $(R, c) = (\ZZ, p)$, $p$: prime. 
    \item [(D)] $(R, c) = (\ZZ[i], 1 + i)$ or
    $(\ZZ[\omega], 1 + \omega)$ where $\omega = e^\frac{2\pi i}{6}$.
\end{itemize}

For $(R, c)$ of type (A) - (C), the following proposition states that we may restrict our consideration only to the above typical cases.

\begin{proposition} \label{prop:red-s-classify}
    Suppose $R$ is a PID and $c \in R$ is prime. 
    \begin{enumerate}
        \item If $(R, c)$ is of type (A), then $\ssr_c = s^{\FF_p}$, where $p = \fchar{R}$. 
        \item If $(R, c)$ is of type (B), then $\ssr_c = s^{\QQ}$. 
        \item If $(R, c)$ is of type (C), then $\ssr_c = \ssr_p(-; \ZZ)$.
    \end{enumerate}
\end{proposition}

We first prove the following two lemmas. 

\begin{lemma} 
\label{lem:lem-classify-1}
    Suppose $(R_0, c_0)$ is another pair such that $R_0$ is a PID and $c_0$ is prime. If there is a ring homomorphism
    \[
        \psi: R_0 \rightarrow R
    \]
    such that $\psi(c_0) \sim c$, then the two corresponding knot invariants coincide,
    \[
        \ssr_{c_0}(-; R_0) = \ssr_c(-; R).
    \]
\end{lemma}

\begin{proof}
    From \Cref{lem:c-div}, we have 
    \[
        \tilde{d}_{c_0}(D; R_0) \leq \tilde{d}_c(D; R)
    \]
    and hence 
    \[
        \ssr_{c_0}(K; R_0) \leq \ssr_c(K; R).
    \]
    From \Cref{thm:1}, both invariants satisfy the mirror formula, so the reverse inequality also holds. 
\end{proof}

\begin{lemma} \label{lem:lem-classify-2}
    Let $R_{(c)}$ denote the localization of $R$ at $c$. The two knot invariants corresponding to $(R, c)$ and $(R_{(c)}, c)$ coincide,
    \[
        \ssr_c(-; R) = \ssr_c(-; R_{(c)}).
    \]
\end{lemma}

\begin{proof}
    Immediate from \Cref{lem:c-localize-at}. 
\end{proof}

\begin{proof}[Proof of \Cref{prop:red-s-classify}]
    %  \Cref{lem:lem-classify-1,lem:lem-classify-2}, it suffices to prove that for each type there is a ring homomorphism
    % \[
    %     \psi: R_0 \rightarrow R_{(c)}
    % \]
    % from the designated pair $(R_0, c_0)$ such that $\psi(c_0) \sim c$.
    If $(R, c)$ is of type (A), there is a ring homomorphism
    \[
        \FF_p[H] \rightarrow R, \quad H \mapsto c
    \]
    and hence $\ssr_c = s^{\FF_p}$. If $(R, c)$ is of type (B), every $n \in \ZZ \setminus 0$ is invertible in $R_{(c)}$. Thus from the universal property of localizations, the ring homomorphism 
    \[
        \ZZ[H] \rightarrow R, \quad H \mapsto c
    \]
    induces
    \[
        \QQ[H] \rightarrow R_{(c)}, \quad H \mapsto c
    \]
    and hence $\ssr_c = s^\QQ$. Finally if $(R, c)$ is of type (C), the unique ring homomorphism
    \[
        \ZZ \rightarrow R
    \]
    sends $p \mapsto p \sim c$ and hence $\ssr_c = \ssr_p(-; \ZZ)$. 
\end{proof}

Thus the invariants $\ssr_c$ that are potentially distinct from the $s$-invariants can be found only in type (C) and (D). Observations on direct computations are given in \Cref{subsec:compu-results}.

\begin{remark}
\label{rem:ZH-theory}
    \Cref{prop:ssr-rasmussen} and \Cref{lem:lem-classify-1} imply that $\ssr_c$ is not slice-torus for a general $(R, c)$, in particular for $(R, c) = (\ZZ[H], H)$. This can be seen as follows: For any $(R, c)$ and consider the ring homomorphism
    \[
        \ZZ[H] \rightarrow R, \quad H \mapsto c.
    \]
    This gives an equality
    \[
        \tilde{d}_H(D; \ZZ[H]) \leq \tilde{d}_c(D; R)
    \]
    and hence 
    \[
        \ssr_H(K; \ZZ[H]) \leq \ssr_c(K; R).
    \]
    If both $\ssr_H$ and $\ssr_c$ are homomorphisms, we get the reverse inequality and hence obtain $\ssr_H = \ssr_c$. This implies that all $s^F$ are equal, but we know from \cite{LS:2014_rasmussen,Seed:KnotKit} that $s^\QQ \neq s^{\FF_2}$, hence a contradiction. More discussion on $(R, c) = (\ZZ[H], H)$ is given in \Cref{subsec:ZH-theory}.
\end{remark}

% Currently we have not found such example. In \Cref{subsec:computations} we list direct computation results for $(R, c) = (\ZZ, 2), (\ZZ, 3)$ and $(\ZZ[i], 1 + i)$. 
    \section{The unreduced counterparts} \label{sec:unreduced}

Here we compare $\ssr_c$ with the unreduced counterpart $\ss_c$ defined in \cite{Sano:2020}. We redefine the quantities as follows:

\begin{definition} \label{def:unred-d-and-s}
    For a link diagram $D$, define
    \[
        d_c(D) = \tilde{d}_c(D_+).
    \]
\end{definition}

\begin{definition}
    For a link $L$, define
    \[
        \ss_c(L) = \ssr_c(L_+) + 1.
    \]
\end{definition}

It follows from \Cref{prop:red-recover-unred} that the definition of $d_c(D)$ coincides with the one given in \cite[Definition 3.5]{Sano:2020}. Also for $\ss_c(L)$, we have
\begin{align*}
    \ss_c(L) 
    &= \ssr_c(L_+) + 1 \\
    &= 2\tilde{d}_c(D_+) + w(D_+) - r(D_+) + 2 \\
    &= 2d_c(D) + w(D) - r(D) + 1
\end{align*}
hence $\ss_c$ coincides with the invariant defined in \cite[Theorem 1]{Sano:2020}%
\footnote{
    As a link invariant, it seems more natural to define $\ss_c(L) = \ssr_c(L_+)$ so that we get $\ss_c(\varnothing) = 0$. The $+1$ is added so as to make $\ss_c(\bigcirc) = 0$.
}. 
With the above redefinition, all of the basic properties of $d_c$ and $\ss_c$ given in \cite{Sano:2020} immediately follows from the results obtained in this paper. 
% In order to compare the reduced and the unreduced invariants, we define the invariant $\epsilon_c$.

\begin{proposition}
    For a pointed link diagram $D$, 
    \[
        \tilde{d}_c(D) \leq d_c(D) \leq \tilde{d}_c(D) + 1.
    \]
    Here, $d_c(D)$ is defined by regarding $D$ as a unpointed link by forgetting the marked point.
\end{proposition}

\begin{proof}
    The injection $\tilde{C}(D) \hookrightarrow C(D)$ maps $\tilde{\ca}(D)$ to $\ca(D)$, and hence
    \[
        \tilde{d}_c(D) \leq d_c(D).
    \]
    On the other hand, from \Cref{prop:d-conn-sum-disj-union},
    \[
        d_c(D) = \tilde{d}_c(D_+) = \tilde{d}_c(D \sqcup \bigcirc) \leq \tilde{d}_c(D) + 1.
    \]
\end{proof}

\begin{definition}
    For a pointed link $L$, define 
    \[
        \epsilon_c(L) = \ss_c(L) - \ssr_c(L).
    \]
\end{definition}

% \begin{example}
%     All positive pointed links are $\epsilon_c$-trivial.
% \end{example}

\begin{proposition}
    $\epsilon_c$ takes values in $\{0, 2\}$.
\end{proposition}

\begin{proof}
    Immediate from 
    \[
        \epsilon_c(L) = 2(d_c(D) - \tilde{d}_c(D)).
    \]
\end{proof}

\begin{corollary}
    If $(R, c)$ makes both $\ssr_c$ and $\ss_c$ satisfy
    \[
        \ssr_c(K) = -\ssr_c(m(K)),\quad \ss_c(K) = -\ss_c(m(K)) 
    \]
    then $\ssr_c = \ss_c$ as knot invariants. \qed
\end{corollary}

\begin{proposition}
\label{prop:unred-eq-red}
    Suppose $R$ is a PID and $c$ is prime. If $(R, c)$ satisfies any one of the following three conditions, then $\ssr_c = \ss_c$ as knot invariants.
    \begin{enumerate}
        \setlength{\itemsep}{0.25em}
        \item $(R, c) = (F[H], H)$ for some field $F$,
        \item $c = 2 \neq 0$,
        \item $2 = 0$.
    \end{enumerate}
\end{proposition}

\begin{proof}
    It suffices to prove that the corresponding invariant $\ss_c$ is an abelian group homomorphism. 
    % For each of these cases, the unreduced complex $C_{h, t}(D)$ splits into the direct sum of two subcomplexes, whose homology have rank $1$ respectively. 
    The first case, with the assumption that $\fchar{F} \neq 2$, is proved in \cite[Section 3.3]{Sano:2020}. This is achieved by splitting the unreduced complex $C_{h, t}(D)$ into two subcomplexes, whose homology have rank $1$ respectively, and describing the generators using the two Lee classes $\ca(D)$ and $\cb(D)$. The other two cases can be proved similarly, where in the third case, one uses a splitting similar to the one given by Wigderson in \cite{Wigderson:2016}.
\end{proof}

\begin{proof}[Proof of \Cref{mainthm:2}]
    Immediate from \Cref{prop:ssr-rasmussen,prop:unred-eq-red}.
\end{proof}
    \section{Computations}
\label{sec:computations}

\subsection{The program}

We developed a computer program that computes from a input knot diagram $D$ the invariants $\ssr_c$ and $\ss_c$ over any supported ring $R = \ZZ,\ \QQ[H],\ \FF_p[H],\ \ZZ[i]$ and $\ZZ[\omega]$ where $\omega = e^{2\pi i/6}$. Our program is capable of computing the invariants $\ss_c$ and $\ssr_c$ for knots with 18-crossing within a few seconds on an ordinary laptop computer. The program is available at 
\begin{center}
    \url{https://github.com/taketo1024/yui}
\end{center}

Our implementation is based on Bar-Natan's efficient algorithm for computing Khovanov homology \cite{Bar-Natan:2007}. The `deloop \& eliminate' method also works for general Khovanov homology $H_{h, t}$ and the reduced $\tilde{H}_{h, t}$. The necessary modification is in the delooping process, where the neck cutting relation and the double-dot removing relation must be replaced by equations (13) and (15) of \cite{Khovanov:2022}. For the reduced case, we must keep track of the pointed component and leave it unlooped until the final step. We also identify the Lee cycle in the simplified chain complex, which is done by applying the same transformations to the original cycle until all circles are deloop. Having obtained the simplified complex and the Lee cycle, we compute its homology in by the standard algorithm (computing the Smith normal form) and vectorize the Lee class. 

For example, the following command computes $\ssr_2$ over $(R, c) = (\ZZ, 2)$ for the trefoil knot $3_1$. 

\begin{verbatim}
    > yui ss 3_1 -t z -c 2 -r
    -2 
\end{verbatim}

\noindent
Internally, the program produces a simplified complex
\[
    \tilde{C} = \{\ 0 \rightarrow \ZZ[-3] \xrightarrow{2} \ZZ[-2] \rightarrow 0[-1] \rightarrow \ZZ[0] \rightarrow 0\ \}
\]
with the Lee cycle vectorized in $\tilde{C}^0 = \ZZ$ as 
\[
    \tilde{\ca} = (-2).
\]
Then the program computes
\[
    d_2 = 1,\ w = -3,\ r = 2
\]
and outputs 
\[
    \ssr_2 = -2.
\]

\subsection{Computational results}
\label{subsec:compu-results}

\begin{table}[t]
    \centering
    \begin{tabular}{c|ccc|cccc}
         Knot 
            & $s^\QQ$ & $s^{\FF_2}$ & $s^{\FF_3}$ 
            & $\ssr_2$ & $\ssr_3$ & $\ssr_{1 + i}$ & $\ssr_{1 + \omega}$ \\
         \hline
         $K14n{19265}$ & 
            0 & {\color{red}-2} & 0 & 
            {\color{red}-2} & 0 & {\color{red}-2} & 0 \\
         $K14n{22180}$ & 
            0 & {\color{red}2} & 0 & 
            {\color{red}2} & 0 & {\color{red}2} & 0 \\
         $K18nh{5566876}$ & 
            2 & 2 & {\color{red}0} & 
            2 & {\color{red}0} & 2 & {\color{red}0} \\
         $K18nh{37144251}$ & 
            -2 & -2 & {\color{red}0} & 
            -2 & {\color{red}0} & -2 & {\color{red}0} \\
    \end{tabular}
    \caption{More computational results}
    \label{tab:computational-results2}
\end{table}

As stated in \Cref{sec:classification}, an invariant $\ssr_c$ that is potentially distinct from the $s$-invariants can be found only for $(R, c)$ of type (C) and (D). Here we examine the cases 
\[
    (R, c) = (\ZZ, 2),\ (\ZZ, 3)
\]
for type (C), and 
\[
    (R, c) = (\ZZ[i], 1+i),\ (\ZZ[\omega], 1+\omega)
\]
for type (D), where $1+i \mid 2$ and $1+\omega \mid 3$. The input data for the following results were those obtained in \cite{Schuetz:2022} and was kindly provided to us.

\begin{proposition} 
\label{prop:comp-results}
    $ $
    \begin{enumerate}
        \item Up to 16 crossings, there are 205 knots such that $s^\QQ = s^{\FF_3} \neq s^{\FF_2}$. For all these knots, we have 
        \[
            s^\QQ = s^{\FF_3} = \ssr_3 = \ssr_{1 + \omega} 
                \ \neq\ 
            s^{\FF_2} = \ssr_2 = \ssr_{1 + i}.
        \]
        \item There are two knots of 18 crossings such that $s^\QQ = s^{\FF_2} \neq s^{\FF_3}$. For these two knots, we have 
        \[
            s^\QQ =  s^{\FF_2} = \ssr_2 = \ssr_{1 + i} 
                \ \neq\ 
            s^{\FF_3} = \ssr_3 = \ssr_{1 + \omega}.
        \]
    \end{enumerate}
\end{proposition}

See \Cref{tab:computational-results2} for some examples of \Cref{prop:comp-results}. The results give rise to the following questions:

\begin{question} \label{ques:1}
    Is $s^{\FF_p} = \ssr_p$ for every prime $p$? 
\end{question}

\begin{question} \label{ques:2}
    Is $\ssr_p = \ssr_c$ for every prime $p$ and every $c \mid p$? 
\end{question}

We also examined whether there are knots such that $\ss_c \neq \ssr_c$. As proved in \Cref{prop:unred-eq-red}, we know that $\ss_c = \ssr_c$ for $(R, c) = (F[H], H)$ and $(\ZZ, 2)$. The next target will be $(\ZZ, 3)$, and we have the following.

\begin{proposition} 
    For the two knots of \Cref{prop:comp-results} (2), we have 
    \[
        \ss_3(K) \neq -\ss_3(m(K)).
    \]
\end{proposition}

This implies that $\ss_3$ is not slice-torus, thus proving \Cref{mainthm:3}. Regarding \Cref{ques:2}, one might expect that $\ss_{1+\omega}$ also satisfies this property, however
\[
    \ss_{1+\omega}(K) = \ss_{1+\omega}(m(K)) = 0,
\]
so $\ss_3 \neq \ss_{1+\omega}$ for the unreduced case.

\subsection{Computations for $(R, c) = (\ZZ[H], H)$}
\label{subsec:ZH-theory}

Finally we observe how $s^F$ and $s^{F'}$ may differ when $\fchar{F} \neq \fchar{F'}$, by considering the Lee class for $(R, c) = (\ZZ[H], H)$ and relating it with the corresponding two cases. Recall from \Cref{rem:ZH-theory} that $\ssr_H(-; \ZZ[H])$ is not slice-torus. Also from the computational aspect, automated homology computation over $\ZZ[H]$ cannot be done since it is not a PID. Nonetheless, the simplification algorithm works over any ring $R$. When the resulting complex is simple enough, the remaining homology computation can be done by hand. 

Here we consider $K = K14n{19265}$ which is the first discovered knot that gives $s^{\QQ} \neq s^{\FF_2}$ \cite{LS:2014_rasmussen,Seed:KnotKit}. 
Below we explain what we obtain by running
\begin{verbatim}
    > yui ckh \
     "[[1,19,2,18],[19,1,20,28],[20,13,21,14],[12,17,13,18],
      [16,21,17,22],[5,15,6,14],[15,5,16,4],[6,27,7,28],
      [2,7,3,8],[26,3,27,4],[25,23,26,22],[11,9,12,8],
      [23,10,24,11],[9,24,10,25]]" -t z -c H -r -a
\end{verbatim}

\medskip

Over $R = \ZZ[H]$, our program simplifies the reduced complex as 
\[
    \tilde{C} = \{\ \cdots \rightarrow R^4[-1] \xrightarrow{d^{-1}} R^4[0] \xrightarrow{d^0} R^2[1] \rightarrow \cdots \},
\]
with differentials
\[
    d^{-1} = \begin{pmatrix}
        0 & -2H & -2 & 2 \\
        0 & -H & 0 & 2 \\ 
        0 & -H^2 & -H & H \\
        0 & 0 & 0 & 0
    \end{pmatrix},\quad
    d^0 = \begin{pmatrix}
        0 & 0 & 0 & H \\
        0 & 0 & 0 & 0
    \end{pmatrix}.
\]
The Lee cycle is vectorized in $\tilde{C}^0 = R^4$ as 
\[
    \ca = (H^4, 0, 0, 0)^T.
\]
\noindent
Now by hand, after some elementary operations, we obtain
\[
    d^{-1} = \begin{pmatrix}
        2 & 0 & 0 & 0 \\
        H & 0 & 0 & 0 \\
        0 & 2 & H & 0 \\
        0 & 0 & 0 & 0
    \end{pmatrix},
\]
with $d^0$ and $\ca$ unchanged. Obviously
\[
    \Ker(d^0) = R^3,
    \quad 
    \Ima(d^{-1}) = 
    \left<\begin{pmatrix}
        2 \\ H \\ 0
    \end{pmatrix}\right>
    \oplus 
    \left<\begin{pmatrix}
        0 \\ 0 \\ 2
    \end{pmatrix},\begin{pmatrix}
        0 \\ 0 \\ H
    \end{pmatrix}\right>
\]
and we get 
\[
    \tilde{H}^0 \ \isom\ 
    R^2/\left<(2, H)^T)\right>\ \oplus\ R/\left<2, H\right>.
\]
Regarding divisibility, we may assume
\[
    \ca = (H^4, 0)^T \in R^2/\left<(2, H)^T)\right>.
\]
By tensoring $\QQ$, we get 
\[
    \QQ[H]^2/\left<(2, H)^T)\right>
    \ \isom\ 
    \QQ[H]
\]
where $\ca$ corresponds to 
\[
    -\frac{H^5}{2}.
\]
Thus $d_H(\ca; \QQ[H]) = 5$ and with $w = -2,\ r = 9$, we get 
\[
    s^\QQ(K) = 0.
\]
On the other hand, by tensoring $\FF_2$ we get 
\[
    \FF_2[H]^2/\left<(0, H)^T)\right>
        \ \isom\ 
    \FF_2[H] \oplus \FF_2
\]
where $\ca$ corresponds to 
\[
    (H^4, 0)^T.
\]
Thus $d_H(\ca; \FF_2[H]) = 4$ and 
\[
    s^{\FF_2}(K) = -2.
\]

    \appendix
\section{Proof of \Cref{prop:cab-reidemeister}}
\label{sec:proof-of-key-prop}

Here we prove \Cref{prop:cab-reidemeister} for each Reidemeister move using the explicit descriptions of the chain homotopy equivalences 
\begin{equation*}
    \begin{tikzcd}
    F:C(D) \arrow[r, shift left] & C(D'):G \arrow[l, shift left]
    \end{tikzcd}
\end{equation*}
given in \cite[Section 4.3]{BarNatan:2004}. Note that the proof also works for \Cref{prop:cab-reidemeister-red} by simply replacing every object with its reduced counterpart.

\subsubsection*{Reidemeister move 1}

See \cite[Figure 5]{BarNatan:2004}. The map $G: C(D') \rightarrow C(D)$ is given by a cap, so it is obvious that $G(\ca(D')) = \ca(D)$ and $G(\cb(D')) = \cb(D)$. On the other hand we have $\delta w(D', D) = \delta r(D', D) = -1$, so the result holds.

\subsubsection*{Reidemeister move 2}

See \cite[Figure 6]{BarNatan:2004}. If the two strands are oriented in the same direction, then the orientation preserving resolutions of $D$ and $D'$ are identical. In this case $G(\ca(D')) = \ca(D)$, $G(\cb(D')) = \cb(D)$ by definition of $G$, while $\delta r(D', D) = 0$. 

If the two strands are oriented in the opposite direction, then the orientation preserving resolution $D'$ yields a circle inside the changing disk. Again by definition of $G$, this circle is capped and a saddle move is performed to the other two strands. The saddle is either a merge or a split, depending on how the strands are connected outside the disk. If it is a merge, then $G(\ca(D')) = \pm c\ca(D)$, $G(\cb(D')) = \mp c\cb(D)$ while $\delta r(D', D) = -2$. If it is a split, then $G(\ca(D')) = \ca(D)$, $G(\cb(D')) = \cb(D)$ while $\delta r(D', D) = 0$. 

\subsubsection*{Reidemeister move 3}

See \cite[Figure 7, 8]{BarNatan:2004}. The equivalences are given by the compositions of equivalences
\[
    \begin{tikzcd}
C(D) \arrow[r, "\tilde{G}", shift left] & C \arrow[l, "\tilde{F}", shift left] \arrow[r, "\tilde{F'}", shift left] & C(D') \arrow[l, "\tilde{G'}", shift left]
    \end{tikzcd}
\]
where $C$ is the mapping cone of two identical maps $\Psi_L = \Psi \circ F_0$ and $\Psi_R = \Psi' \circ F'_0$. It suffices to prove that the desired relations hold in $C$ under the maps $\tilde{G}$ and $\tilde{G}'$. If the center crossing is negative, then $\ca(D)$ and $\cb(D)$ lies in the codomain of $\Psi$ and are mapped identically into $C$. The same is true for $\ca(D')$ and $\cb(D')$ while  $\delta(D, D') = 0$ so the result holds. If the center crossing is positive, then $\ca(D)$ and $\cb(D)$ lies in the domain of $\Psi$, and the cycles are mapped as
\[
    \begin{pmatrix} z \\ 0 \end{pmatrix}
        \xmapsto{\tilde{G}}
    \begin{pmatrix} G_0 z \\ \Psi h_0 z \end{pmatrix}.
\]
Here the maps $G_0, h_0$ are the maps corresponding to the RM2 performed on the other two crossings. When $h_0(z) = 0$ the result reduces to the case of RM2. The only case where $h_0(z) \neq 0$ is when the cycle $z$ belongs to the state that contains a circle in the changing disk, which happens only when the three strands are oriented symmetrically with respect to the rotation by $\pi/3$. However, by a sequence of moves described in the beginning of \cite[Section 3]{Polyak:2010}, we may avoid this move and assume that $h_0 = 0$. This deformation is valid since we know from \cite[Theorem 4]{BarNatan:2004} that the cobordism maps are invariant (up to sign) under isotopies. Thus the proof is complete.

    \printbibliography
    \addresses
\end{document}